\documentclass{lms}
\usepackage{amssymb,amsmath,color}
\usepackage[all]{xy} 
\usepackage{mathrsfs}
\usepackage{multirow}

\newtheorem{Def}{Definition}[section]
\newtheorem{Prop}[Def]{Proposition}
\newtheorem{Theo}[Def]{Theorem}
\newtheorem{Lem}[Def]{Lemma}
\newtheorem{Koro}[Def]{Corollary}

\DeclareMathOperator{\Coker}{{Coker}}
\DeclareMathOperator{\Ker}{{Ker}}

\DeclareMathOperator{\rad}{{rad}}

\DeclareMathOperator{\add}{{add}}
\DeclareMathOperator{\thick}{{\sf thick}}

\newcommand{\D}[1]{\mathscr{D}(#1)}

\newcommand{\Db}[1]{ \mathscr{D}^{\rm b}(#1)}

\newcommand{\C}[1]{\mathscr{C}(#1)}

\newcommand{\K}[1]{\mathscr{K}(#1)}

\newcommand{\Kb}[1]{ \mathscr{K}^{\rm b}(#1)}

\newcommand{\modcat}[1]{#1\mbox{{\sf -mod}}}
\newcommand{\Modcat}[1]{{#1{\sf \mbox{-}Mod}}}

\newcommand{\pmodcat}[1]{#1\mbox{{\sf -proj}}}
\newcommand{\imodcat}[1]{#1\mbox{{\sf -inj}}}
\newcommand{\stp}[1]{{#1}\mbox{\sf-stp}}
\def\nuinv#1{\mathscr{X}_{#1}}
\def\Ndomdim{\nu\mbox{-}\domdim}

\newcommand{\Hom}{{\rm Hom }}

\newcommand{\RHom}{{\rm\bf R}{\rm Hom}}

\newcommand{\End}{{\rm End}}
\newcommand{\Ext}{{\rm Ext}}

\newcommand{\opp}{^{\rm op}}
\newcommand{\otimesL}{\otimes^{\rm\bf L}}

\DeclareMathOperator{\domdim}{domdim}
\DeclareMathOperator{\projdim}{pd}

\DeclareMathOperator{\gldim}{gldim}

\DeclareMathOperator{\DD}{\mathrm{D}}  

\newcommand{\cpx}[1]{#1^{\bullet}}

\renewcommand{\H}{\mathcal{H}}

\renewcommand{\leq}{\leqslant}
\renewcommand{\geq}{\geqslant}

\newcommand{\lra}{\longrightarrow}
\newcommand{\ra}{\rightarrow}
\newcommand{\lraf}[1]{\stackrel{#1}{\lra}}
\newcommand{\raf}[1]{\stackrel{#1}{\ra}}

\title[Derived equivalences, restriction and invariance]{Derived equivalences, restriction to self-injective subalgebras and
  invariance of homological dimensions}
\author{Ming Fang, Wei Hu$^*$ and Steffen Koenig}
\date{}
\classno{18E30, 16E10, 16G10, 16L60, 20G43}

\extraline{M.Fang and W.Hu are both partially supported by NSFC (11471315, 11321101, 11331006, 11471038). W.Hu is grateful to the Fundamental Research Funds for the Central Universities for partial support.
}

\begin{document}
\maketitle

\setcounter{footnote}{-1} \footnote{ $^*$ Corresponding author.}



\begin{abstract}

Derived equivalences between finite dimensional algebras do, in general,
  not pass to centraliser (or other) subalgebras, nor do they preserve
  homological invariants of the algebras, such as global or dominant
  dimension. We show that, however, they do so for large classes of
  algebras described in this article.

  Algebras $A$ of $\nu$-dominant dimension at least one have unique largest
  non-trivial self-injective centraliser subalgebras $H_A$.
  A derived restriction theorem
  is proved: A derived equivalence between $A$ and $B$ implies a derived
  equivalence between $H_A$ and $H_B$.

   Two methods are developed
   to show that global
  and dominant dimension are preserved by derived equivalences between
  algebras of $\nu$-dominant dimension at least one with anti-automorphisms
  preserving simples, and also between almost self-injective algebras.
  One method is based on identifying particular derived equivalences
  preserving homological dimensions, while
  the other method identifies homological dimensions inside certain derived
  categories.

   In particular, derived
  equivalent cellular algebras have the same global dimension.
  As an application, the global and dominant
  dimensions of blocks of quantised Schur algebras
  with $n \geq r$ are completely determined.

\end{abstract}

\tableofcontents

\section{Introduction}

Derived equivalences between finite dimensional algebras are known to be
fundamental in representation theory and applications. Unfortunately, still
very few positive results are known about the structure of derived equivalences and
about homological invariants. For instance, it is not known when (a)
a derived equivalence between algebras $A$ and
$B$ induces derived equivalences between certain centraliser subalgebras $eAe$ and
$fBf$, or in case of group algebras between subgroup algebras. It is also not
known when (b) the existence of a derived equivalence implies that $A$ and
$B$ share homological invariants such as global or dominant dimension.
For known classes of derived equivalences, both questions are known to have
dauntingly negative answers.


The aim of this article is to identify large classes of algebras where both
problems do have positive solutions. A starting point, and some hope, may be
provided by the class of self-injective algebras, which have both global and
dominant dimension infinite, except in the semisimple case. Under some mild
assumption, derived equivalences are known to preserve the property of being
self-injective, hence also the values of these two homological dimensions.
In full generality, derived equivalences between self-injective algebras $A$ and $B$
induce stable equivalences of Morita type and thus there are equalities
between global and also between dominant dimensions of $A$ and $B$. This
motivates considering algebras that are closely related to self-injective algebras
and considering derived equivalences inducing stable equivalences of Morita type,
 leading to the following more precise versions of (a) and (b):

\medskip
 {\em
 {\rm (A)} When does a derived equivalence between algebras $A$ and $B$
 induce a derived equivalence between (largest, in some sense)
 self-injective centraliser subalgebras
$eAe$ and $fBf$?

{\rm  (B)} For which classes of algebras are
global and dominant dimension invariant under all or certain derived
equivalences?}

\medskip
The key concept to address both questions and to identify suitable and interesting classes of algebras is $\nu$-dominant
dimension (to be defined in \ref{defn:strong dominant dimension}),
where $\nu$ is the Nakayama functor sending projective to
injective modules over an algebra. Concerning (A), an assumption is
needed to identify unique and non-zero associated self-injective centraliser
subalgebras (first used by Martinez-Villa \cite{MV1}) of $A$
and $B$, respectively, which then can be investigated for derived equivalence.
In general, it may happen that an algebra $A$ having zero as associated
self-injective algebra is derived equivalent to an algebra $B$
having a non-zero associated self-injective algebra; see
\cite[Section 5]{Rickardderivedstable} for an example.

The condition that both algebras have faithful strongly projective-injective
modules allows to identify non-trivial associated self-injective
centraliser subalgebras (Lemma
\ref{lem:two dominant dimensions coincide} and Definition
\ref{defn:associatedselfinjectivealgebra}), and it is strong enough to solve
problem (A):

\medskip
{\bf Derived Restriction Theorem (Corollary
  \ref{cor:restriction theorem for Morita algebras})}:
{\em Let $A$ and $B$ be finite dimensional algebras
of $\nu$-dominant dimension at least one, and let $H_A=eAe$ and $H_B=fBf$ be
their associated self-injective centraliser subalgebras. If $A$
and $B$ are derived equivalent, then also $H_A$ and $H_B$ are derived
equivalent.}

\medskip
The proof is based on a stronger result (Theorem
\ref{thm:restriction theorem}), which shows that the given derived
equivalence between $A$ and $B$ restricts to certain subcategories
that are shown to determine the derived categories of $H_A$ and $H_B$.

The class of algebras of $\nu$-dominant dimension
at least one contains all self-injective
algebras, but also the Morita algebras introduced in \cite{KY13}, which are
characterised in (Theorem \ref{prop:characterise morita algebras}) as the algebras having $\nu$-dominant dimension at least two; their $\nu$-dominant dimension coincides with the classical dominant dimension. Morita algebras in turn contain gendo-symmetric algebras and hence several classes
of algebras of interest in algebraic Lie theory such as classical or
quantised Schur algebras and blocks of the BGG-category $\mathcal O$ of
semisimple complex Lie algebras; these algebras usually have finite global
dimension, but are related to self-injective algebras by Schur-Weyl dualities.
Special cases of the Derived Restriction Theorem state for instance: (1) Two
classical or quantised Schur algebras $S(n,r)$ (with $n \geq r$)
are derived equivalent only if the
corresponding group algebras of symmetric groups or Hecke algebras are so
(for the latter a derived equivalence classification is known by Chuang and
Rouquier \cite{CR08}). \\ (2) Auslander algebras of self-injective algebras
of finite representation type are derived equivalent if and only if the self-injective algebras are so (the latter derived equivalences are known by work of
Asashiba \cite{Asa}), and in this case the Auslander algebras moreover are
stably equivalent of Morita type (Corollary
\ref{cor:restriction theorem for Auslander algebras}).

Problem (B) does not, in general, have a positive answer for algebras of
$\nu$-dominant dimension at least one; only upper bounds for the
differences in dimensions can be given (which are valid in general, Proposition  \ref{prop:bound global dimension} and Theorem \ref{thm:bound the difference}). To identify subclasses of
algebras where problem (B) has a positive answer, two approaches are
developed here:

One approach identifies special derived equivalences, which do preserve both
global and dominant dimensions: these are the (iterated) almost $\nu$-stable
derived equivalences (introduced in  \cite{HuXi10,Hu12}). These equivalences are
known to induce stable equivalences of
Morita type. Standard equivalences between self-injective algebras are of
this form. The same is true for a larger class of algebras introduced here, the almost
self-injective algebras, which include all self-injective algebras and also some
algebras of finite global dimension, for instance Schur algebras of finite
representation type:

\medskip
{\bf First Invariance Theorem (Corollary \ref{cor:class almost self-injective algebras}):}
{\em Derived equivalences between almost
self-injective algebras preserve both global and dominant dimension.}

\medskip
The second approach concentrates on directly identifying global and dominant
dimension inside some derived category; in the case of dominant dimension, the
associated self-injective centraliser subalgebra occurring in the Derived
Restriction Theorem is used. This approach works (under the assumption of
having dominant dimension at least one) for all split algebras
(e.g.,
algebras over an algebraically closed field) having an
anti-automorphism (for instance, a duality) preserving simples:

\medskip
{\bf Second Invariance Theorem
(Theorem \ref{thm:class with anti-automorphisms}):}
{\em Let $A$ and $B$ be two derived equivalent split algebras with
anti-auto\-morphisms fixing simples. Then they have the same
global dimension. If in addition both $A$ and $B$ have dominant dimension at
least one, then they also
have the same dominant dimension.}

\medskip
Dualities, i.e. involutory anti-automorphisms fixing simples,
exist for instance for all cellular algebras.
The second invariance theorem covers in particular classical and quantised
Schur algebras and even their blocks. In fact, the invariance property
is strong enough
(Theorems \ref{thm:global dimension of blocks of Schur algebras}
and \ref{thm:dominant dimension of blocks of Schur algebras}) to
determine these dimensions by explicit combinatorial formulae in terms of
weights and (quantum) characteristics,
for all blocks of such Schur algebras,
using the derived equivalences constructed by Chuang and Rouquier.

The main results of this article are motivated by various results in the
literature, which are extended and applied here:

The idea to use Schur-Weyl dualities and Schur functors to compare
homological data of self-injective algebras such as group algebras of
symmetric groups and of their Schur algebras (quasi-hereditary covers), of
finite global dimension, has been developed in \cite{FK11a}. There it has been
demonstrated that dominant dimension is not only crucial for existence of
Schur-Weyl duality, \cite{KSX}, but also for the quality of the Schur functor
in preserving homological data, although typically on one side of Schur-Weyl
duality there is an algebra of infinite global and dominant dimension and on
the other side there is an algebra with finite such dimensions. These
are derived equivalent only in degenerate (semisimple) cases. The same
approach has been demonstrated to work for Hochschild cohomology in
\cite{FM}, where also a first instance of the derived restriction theorem has
appeared. The concept of (iterated) almost $\nu$-stable derived equivalence
and its useful properties have been developed in \cite{Hu12,HuXi10}, where the
focus has been on the resulting stable equivalences of Morita type (which
imply invariance of global and dominant dimension).

Chuang and Rouquier's derived equivalence classification of blocks of
symmetric groups and of some related algebras provides an important supply of
examples. It turns out, however, that already these equivalences, more
precisely those between quantised Schur algebras, are not always iterated
$\nu$-stable, which forces us to use the second approach to derived
invariants, using anti-automorphisms fixing simples, in this case.


\section{Three homological invariants and two classes of algebras}
\label{sec:algebras}

After recalling two major homological invariants of algebras, global dimension and dominant dimension, a new variation of dominant dimension, $\nu$-dominant dimension, is introduced that will turn out to provide a crucial assumption in the main results. In the second subsection, the two main classes of algebras considered here will be discussed and related to $\nu$-dominant dimension; these
are the Morita algebras, which will get characterised in terms of
$\nu$-dominant dimension, and the new class of almost self-injective algebras.

\vspace{-.2cm}

\subsection{General conventions}
 Throughout, $k$ is an arbitrary field of any characteristic. Algebras are
{\em finite dimensional} $k$-vector spaces and, unless stated otherwise,
 modules are finitely generated left modules.
 When $A$ is an algebra, $A\opp$ denotes the opposite algebra of $A$, and
 $A^{\mathrm e}$ is the enveloping algebra $A\otimes_k A\opp$.
 Let $\Modcat{A}$ (respectively $A\modcat$) be the category of all
 (respectively all finitely generated) left $A$-modules, and
 $\pmodcat{A}$ (respectively $\imodcat{A}$)
 the full subcategory of $A\modcat$ whose objects are the
 projective (respectively injective) left $A$-modules.
 Let $\DD$ be the usual $k$-duality functor
 $\Hom_k(-,k): \modcat{A} \to A\opp\modcat$ and $\nu_{_A} = \DD\Hom_A(-,A):
 \pmodcat{A}\to \imodcat{A}$ the Nakayama functor.

 We follow the conventions from \cite{AF}. Let $\cal C$ be an
 additive category.  An object  $X$ in $\cal C$ is called
 \emph{strongly indecomposable} if $\End_{\cal C}(X)$ is a local
 ring. An object $Y$ in $\cal C$ is called \emph{basic}
 if $Y$ is a direct sum of strongly indecomposable objects of multiplicity
 one each. For an object $M$ in $\cal C$, we write $\add(M)$
 for the full subcategory of $\cal C$
 consisting of all direct summands of finite direct sum of copies of $M$.
 By $f \cdot g$ or $f g$ we denote the composition of
 morphisms $f: X\to Y$ and $g: Y\to Z$ in $\cal C$.
 A morphism $h: X\to Y$ is said to be \emph{radical} in $\cal C$,
 if for morphisms $\alpha: Z\to X$ and $\beta: Y\to Z$, the composition
 $\alpha\cdot h\cdot \beta$ never is an isomorphism.
 In contrast to the composition rule for morphisms, we write
 $\mathcal{G}\circ \mathcal{F}$ for the composition of two functors
 $\mathcal{F}: \cal C\to \cal D$ and $\mathcal{G}: \cal D
 \to \cal E$ between additive categories.



\vspace{-.2cm}

 \subsection{Global dimension and two dominant dimensions}
\label{subsec:global dimension and dominant dimension}
 Given a finite dimensional $k$-algebra $A$, there are many homological
 invariants around to measure the complexity of $A$ from different points
 of view, and global dimension is the most widely used one.
 By definition, the global dimension of $A$, denoted by $\gldim A$, is
 the smallest number $g$
 or $\infty$ such that $\Ext^i_A(M,N)=0$ for any $i>g$ and all $M,N\in \modcat{A}$.
 The following (well-known) characterisation can be found for instance in \cite[Corollary 3.8]{FK11a}.

\vspace{-.2cm}

 \begin{Lem} \label{lem:characterise global dimension}
   Let $A$ be a $k$-algebra. If $\gldim A<\infty$, then
   $\gldim A$ is the largest number $g$ such that $\Ext^g_A({_A}\DD(A),{_A}A)
   \neq 0$.
 \end{Lem}

\vspace{-.2cm}

 Dominant dimension was introduced by Nakayama,
 and developed later mainly by Morita and Tachikawa, see \cite{Y96}
 for more information, and \cite{FM,FK11a,FK11b,FK15,FKY} for a recent
 development partly motivating our aims and results.

 \begin{Def}\label{defn:dominant dimension}
   Let $A$ be a $k$-algebra. The {\em dominant dimension} of $A$,
   denoted by $\domdim(A)$, is defined to be
   the largest number $d \geq 0$ (or $\infty$) such that in
   a minimal injective resolution $0\to {_A}A\to I^0\to I^1\to I^2\to \cdots$
   of the left regular $A$-module, $I^i$ is projective for all $i<d$ (or $\infty$).
 \end{Def}

 Thus, $I^0$ not being projective, that is $\domdim(A)= 0$, is equivalent to $A$
 not having a faithful projective-injective module.
 The module $I\in \modcat{A}$ is projective and injective if and only
 if so is $\DD(I)$ in $\modcat{A\opp}$. It follows that $\domdim (A) =
 \domdim (A\opp)$ and thus $\domdim(A)$ can be defined alternatively
 via right $A$-modules.
 If $\domdim(A)\geq 1$, then there exists a
 unique (up to isomorphism) minimal faithful right $A$-module (and also
 a unique up to isomorphism minimal faithful left $A$-module). It must be
 projective and injective, hence of the form $eA$ for some
 idempotent $e$ in $A$. If $\domdim(A)\geq 2$, then $eA$ is a
 faithful balanced bimodule, i.e., there is a double centralizer
 property, namely $A\cong \End_{eAe}(eA)$ canonically. Algebras of
 infinite dominant dimension are conjectured to be self-injective
 (this is the celebrated \emph{Nakayama conjecture}), see \cite{Y96}.
 The following characterisation of dominant dimension is due to
 M\"{u}ller.

\vspace{-.2cm}
 \begin{Prop}[(M\"uller \cite{Mu68})] \label{prop:Muller characterisation}
 Let $A$ be a $k$-algebra of dominant dimension at least $2$.
 Let $eA$ be a minimal faithful right $A$-module and $n\geq 2$ be an integer.
 Then $\domdim(A)\geq n$ if and only if $\Ext_{eAe}^i(eA,eA)=0$ for $1\leq i\leq n-2$.
 \end{Prop}

\vspace{-.2cm}
Of particular interest later on will be certain derived equivalences,
(iterated) almost $\nu$-stable derived equivalences, defined in
\cite{Hu12,HuXi10}. Here, a certain subclass of both projective and injective
 (projective-injective for short) modules is crucial. This motivates the
 following variation of dominant dimension, which is crucial for our main results.

\vspace{-.2cm}
 \begin{Def}\label{defn:strong projective injective modules}
   Let $A$ be a $k$-algebra and $\nu_{_A}=\DD\Hom_A(-,A):
   \modcat{A}\to \modcat{A}$ be the Nakayama functor.
   A projective $A$-module $P$ is said to be
{\em strongly projective-injective} if $\nu_{_A}^i(P)$ is projective for
   all $i>0$.
   By $\stp{A}$ we denote the full subcategory of $\pmodcat{A}$ consisting of
   strongly projective-injective $A$-modules.
 \end{Def}


\vspace{-.2cm}
In \cite{HuXi}, strongly projective-injective modules are called
$\nu$-stably projective; since this may be misunderstood as implying
$\nu$-stable, we use a different terminology here.

 Strongly projective-injective modules are
 projective and injective, which justifies their name.
 An easy proof goes as follows, see
 also \cite[Lemma 2.3]{HuXi}. First note that $P$
 is strongly projective-injective if and only if so is each
 of its direct summands. Thus we may assume that $P$ is
 indecomposable. Since the Nakayama functor $\nu_{_A}$ sends indecomposable
 projective modules to indecomposable injective modules,
 it follows that $\nu^i_{_A}(P)$ are indecomposable projective-injective
 for all $i>0$. But there are only finitely
 many indecomposable objects in $\pmodcat{A}$, so there must
 exist $0<a<b$ such that $\nu_{_A}^a (P)\cong \nu_{_A}^b(P)$.
 Using again that $\nu_{_A}: \pmodcat{A}\to \imodcat{A}$ is an equivalence,
 we deduce that $P\cong \nu_{_A}^{b-a}(P)$. In particular, $P$ is both
 projective and injective.
 \begin{Def}\label{defn:strong dominant dimension}
   Let $A$ be a $k$-algebra. The {\em $\nu$-dominant
     dimension}
   of $A$, denoted by $\Ndomdim(A)$, is defined
   to be the largest number $d \geq 0$ (or $\infty$) such that in a minimal
   injective resolution $0\to {_A}A\to I^0\to I^1\to I^2\to \cdots$
   of the left regular $A$-module, $I^i$ is strongly projective-injective
   for all $i<d$ (or $\infty$).
 \end{Def}

 By definition, $\Ndomdim(A)\leq \domdim(A)$, but in general there is no equality. Here is an example
 illustrating the difference between these two dimensions.
 Let $A$ be the path algebra $kQ$ of the quiver $Q: 1\lra 2$. Then $P_1$
 is projective-injective, but not strongly projective-injective, since
 $\nu_{_A}(P_1)\cong \mathrm{D}\Hom_{A}(P_1,A)\cong \mathrm{D}(e_1A)\cong I_1$
 and $I_1$ is not projective. As a result, $\domdim(A)=1$, while $\Ndomdim(A)=0$.

 In our context, $\nu$-dominant dimension is important, since it allows to identify
 particular self-injective centraliser subalgebras:

\begin{Lem}\label{lem:two dominant dimensions coincide}
   Let $A$ be a $k$-algebra. If $\Ndomdim(A)\geq 1$, then
   all projective-injective $A$-modules are strongly
   projective-injective, and thus $\Ndomdim(A)=\domdim(A)$.
   In this case, endomorphism rings of minimal faithful
 left $A$-modules are self-injective.
 \end{Lem}

Suppose $\Ndomdim(A)\geq 1$.
Then, a minimal faithful left $A$-module is of the form $Ae$, and
strongly projective-injective. We will use its endomorphism ring $eAe$
as `the largest self-injective centraliser subalgebra'.

 \begin{proof}
 Since $\Ndomdim(A)\geq 1$, the injective envelope
 $I$ of ${}_AA$ is strongly projective-injective.
 Let $P$ be an indecomposable projective-injective $A$-module. The composition
 \mbox{$P\hookrightarrow A\hookrightarrow I$} is a split monomorphism. Thus $P$
 is a direct summand of $I$, and in particular strongly projective-injective.
 Consequently, all projective-injective $A$-modules are strongly
 projective-injective. Hence the two dominant dimensions coincide.

 Let $Ae$ be a minimal
 faithful left $A$-module.
 By assumption, it is strongly projective-injective. Hence
 $\DD(eA) \cong \nu_{_A}(Ae)$ belongs to $\add(Ae)$,
 and in particular \mbox{${_{eAe}}\DD(eAe) = e \DD(eA)$} 
$\in e\add{({_A}Ae)} = \add{{_{eAe}}(eAe)}$,
 that is, $eAe$ is self-injective.
 \end{proof}

The endomorphism ring of a strongly projective-injective $A$-module
in general may not be self-injective, even when assuming
$\Ndomdim(A)\geq 1$. For instance, let $A$ be the self-injective
Nakayama algebra with cyclic quiver, three simple modules and $\rad(A)^2=0$.
Then each indecomposable projective module is injective as well, and even
strongly projective-injective, but the endomorphism ring of a sum of two
non-isomorphic indecomposable projective modules never is self-injective.

The proof of Lemma \ref{lem:two dominant dimensions coincide} works not
only for $Ae$, but also for any direct sum of copies of $Ae$.

\begin{Def} \label{defn:associatedselfinjectivealgebra}
Let $A$ be a $k$-algebra with $\Ndomdim(A)\geq 1$ and let $Ae$ be a minimal
faithful left $A$-module. Then the centraliser algebra $eAe$ is called
the {\em associated self-injective algebra of $A$}.
\end{Def}


These associated self-injective
subalgebras have been introduced and strongly used before in work of
Martinez-Villa \cite{MV1,MV2} reducing validity of the Auslander-Reiten
conjecture on stable equivalences preserving the number of non-projective
simple modules to the case of self-injective algebras.
There, the setup is more general and the associated self-injective algebras
have been allowed to be zero, which does not make sense in our situation
as we need a strong connection between the given algebra and its associated
self-injective subalgebra. The term `associated self-injective algebra'
first occurred in \cite{DugasMV}.

\subsection{Morita algebras and almost self-injective algebras} \label{subsec:Morita algebras}
The term `Morita algebras' (not related to Morita rings occurring in Morita contexts)
was coined by Kerner and Yamagata in \cite{KY13}, when they investigated algebras
 first studied by Morita \cite{Mo58} as endomorphism rings
 of generators over self-injective algebras.
 The subclass of Morita algebras consisting of endomorphism rings
 of generators over symmetric algebras, called gendo-symmetric algebras,
 was introduced and studied independently in \cite{FK11b,FK15}.

 \begin{Def}[\cite{KY13,FK11b,FK15}] \label{def:morita algebras}
   A $k$-algebra $A$ is called a {\em Morita algebra}
   if $A$ is isomorphic to \mbox{$\End_H(H\oplus M)$}
   for some self-injective algebra
   $H$ and some module $M\in \modcat{H}$.
   $A$ is called {\em gendo-symmetric} if in addition
   $H$ is symmetric.
 \end{Def}

 Gendo-symmetric algebras form a large class of algebras, cutting across
 traditional boundaries such as finite or infinite global dimension. Examples
 of finite global dimension
 include classical and quantised Schur algebras $S(n,r)$ (with $n \geq r$),
 blocks of the Bernstein-Gelfand-Gelfand category $\mathcal O$ and many other
 algebras occurring in algebraic Lie theory and elsewhere.
 Examples of infinite global dimension include symmetric
 algebras, for instance group algebras and Hecke algebras.
 Morita algebras include in addition self-injective,
 and in particular Frobenius algebras, and their Auslander algebras.
Morita algebras have been characterised in several ways, see
\cite{KY13,FK15,FKY}.
 For our purposes, a new characterisation is needed in terms of $\nu$-dominant
 dimension (Section \ref{subsec:global dimension and dominant dimension}):

\vspace{-.2cm}
 \begin{Prop}\label{prop:characterise morita algebras}
   Let $A$ be a $k$-algebra. Then $\Ndomdim(A)\geq 2$ if and only
   if $A$ is a Morita algebra.
 \end{Prop}
 \begin{proof}
Suppose $\Ndomdim(A)\geq 2$. Then by
Lemma \ref{lem:two dominant dimensions coincide}, the minimal faithful
$A$-module $Ae$ is (strongly)
projective-injective and its endomorphism ring $eAe$ is
self-injective. Moreover, $A$ has dominant dimension at least two, which
implies a double centraliser property on $Ae$, between $A$ and $eAe$.
Therefore, $A$ is a Morita algebra.


  Conversely, if $A$ is a Morita algebra, then $A$ is isomorphic to $\End_H(M)$
  for some self-injective algebra $H$ and a generator $M$ in $\modcat{H}$.
  Let $E$ be the direct sum of all pairwise non-isomorphic indecomposable projective
  $H$-modules.

  \smallskip
  {\em Claim.  $\Hom_H(M, E)$ is a strongly  projective-injective (left) $A$-module. }

 \smallskip
 {\em Proof.}
   By definition, $E$ is a direct summand of $M$ and $\add({}_HE)=\pmodcat{H}$.
  Moreover, \mbox{$\nu_{_H}E\cong E$} 
as left $H$-modules and there are isomorphisms of $A$-modules
  \begin{align*}
    \nu_{_A}\Hom_H(M, E) & = \DD\Hom_A\big(\Hom_H(M, E), \Hom_H(M, M)\big) \\
    & \stackrel{(1)}{\cong} \DD\Hom_H(E, M) \stackrel{(2)}{\cong} \DD(\Hom_H(E,H)\otimes_H M) \\
    &\stackrel{(3)}{\cong} \Hom_H(M, \nu_{_H}E)
    \stackrel{(4)}{\cong} \Hom_H(M, E).
  \end{align*}
  Here, the isomorphism $(1)$ follows from $E\in \add{M}$. The isomorphism 
$(2)$ uses \mbox{$\add({}_HE)=\pmodcat{H}$} 
and the isomorphism $(3)$ follows from
  tensor-hom adjointness. $(4)$ uses $\nu_{_H}E\cong E$. This proves the claim.

\medskip
  Now we construct an injective presentation of the left regular $A$-module
  ${_A}A$ (or $\End_H(M)$) as follows:
  take an injective presentation $0\ra {}_HM\ra P_1\ra P_2$ of $M$
  and apply $\Hom_H(M, -)$ to obtain the exact sequence
  $0\ra \Hom_H(M, M)\ra \Hom_H(M, P_1)\ra\Hom_H(M, P_2)$ of left 
\mbox{$A$-modules.}
  Note that both $P_1$ and $P_2$ are projective $H$-modules and thus belong to
  $\add({}_HE)$. Therefore, $\Hom_H(M, P_i)\in \stp{A}$ for $i=1,2$, and
  so $\Ndomdim(A)\geq 2$.
 \end{proof}

 \begin{Koro}
   Let $A$ be a Morita algebra. Then $\domdim(A)=\Ndomdim(A)$.
 \end{Koro}
 \begin{proof}
   This follows immediately from Lemma \ref{lem:two dominant dimensions coincide} and
   Proposition \ref{prop:characterise morita algebras}.
 \end{proof}

Gendo-symmetric algebras appeared first in \cite{FK11b}, see
 \cite{FK15} for further information.
 In our context, there is the following characterisation:

 \begin{Prop}[(\cite{FK11b,FK15})]\label{prop:properties on gendo-symmetric algebras}
  Let $A$ be a $k$-algebra.
  Then $A$ is gendo-symmetric if
  and only if $\DD(A)\otimes_A \DD(A)\cong \DD(A)$ as $A$-bimodules.
  If $A$ is gendo-symmetric, then $\domdim(A) = \sup\{s\mid \Ext^i_A(\DD(A),A)=0,
  \; 1\leq i\leq s-2\}$.
 \end{Prop}

 \begin{proof}
   The first claim follows from \cite[Theorem 3.2]{FK11b} and the second one
   from \cite[Proposition 3.3]{FK11b}. 
   Alternatively, the characterisation of $\domdim(A)$
   also follows from Proposition \ref{prop:Muller characterisation} combined 
   with the following claim:

\smallskip
   {\em Claim. Let $e$ be an idempotent in $A$ such that $eA$
   is a minimal faithful right $A$-module. Then there are
   isomorphisms $\Ext^i_A(\DD(A),A)\cong \Ext^i_{eAe}(eA,eA)$ in $\modcat{A}$ 
   for \mbox{$0\leq i\leq \domdim(A)-1$.}}

\smallskip
   {\em Proof.}
    Since $A$ is a gendo-symmetric algebra,
   dualising both sides of the isomorphism
${\DD(A)\otimes_A \DD(A)} \cong \DD(A)$
   yields an isomorphism $\Hom_A(\DD(A),A)\cong A$ as $A$-bimodules.
   Thus $\Hom_A(\DD(A),Ae)$ is isomorphic to $Ae$, and in particular it is projective-injective.
   Now, for a minimal injective resolution
   $
   \mathscr{E}:\; 0\to {_A}A\to I^0\to I^1\to  I^2\to \cdots
   $
   of the left regular $A$-module, the first $\domdim(A)$-terms
   are projective-injective, hence belong to $\add{({_A}Ae)}$.
   Applying $\Hom_A(\DD(A),-)$ to the sequence $\mathscr{E}$,
   and comparing the cohomologies of $\Hom_A(\DD(A),\mathscr{E})$
   and $\Hom_{eAe}(eA,e\mathscr{E})$, proves the claim. 
 \end{proof}

The second class of algebras we are going to study generalises self-injective
algebras.

\begin{Def}
  An algebra $A$ is called an {\em almost self-injective algebra},
  if $\Ndomdim(A)\geq 1$ and there is at most one indecomposable
  projective $A$-module that is not injective.
\end{Def}

Among the examples are Schur algebras of finite representation type, which
have finite global dimension and which are Morita algebras. Schur algebras,
and thus Morita algebras, in general are not almost self-injective. 
Schur algebras of finite representation type are examples of gendo-Brauer tree algebras described and classified in \cite{ChM}. These algebras are representation-finite gendo-symmetric and in addition biserial; the corresponding symmetric algebras are Brauer tree algebras. Conversely, almost self-injective
 algebras need not be Morita algebras as the following example illustrates.
 Let $A$ be the $k$-algebra given by the quiver

 \[
    \xymatrix{2 \ar@/^/[r]^\delta & 1 \ar@/^/[l]^\alpha \ar@/^/[r]^\beta & 3 \ar@/^/[l]^\theta}
 \]
 and relations $\{\delta\alpha, \alpha\delta, \theta\beta, \beta\theta\}$.
 The Loewy series of the indecomposable projective left $A$-modules are
 \[
 \xymatrix@C=.5mm@R=.5mm{ \\ P_1= \\} \xymatrix@C=.5mm@R=.5mm{&1&\\2 & &3 \\}\hspace{1cm}
 \xymatrix@C=.5mm@R=.5mm{ \\ P_2= \\} \xymatrix@C=.5mm@R=.5mm{2\\1\\3}\hspace{1cm}
 \xymatrix@C=.5mm@R=.5mm{ \\ P_3= \\} \xymatrix@C=.5mm@R=.5mm{3\\1\\2}
 \]
 Then both $P_2$ and $P_3$ are strongly
 projective-injective, and $\Ndomdim(A)=1$. Thus $A$ is an almost self-injective
 algebra, but not a Morita algebra.

 \section{Derived equivalences}
 \label{sec:derivedequivalences}

 After recalling fundamental facts of derived Morita theory, basic properties
 of standard equivalences will be shown and then almost $\nu$-stable derived
 equivalences will be explained, thus providing crucial tools for proofs
 later on.

 Let $\cal C$ be an additive category. A
 complex $\cpx{X}$ over ${\cal C}$ is a sequence of morphisms $d_X^i$
 in $\cal C$ of the  form
 $$\cdots\longrightarrow X^{i-1}\xrightarrow{d_X^{i-1}}X^i\xrightarrow{d_X^i}X^{i+1}\xrightarrow{d_X^{i+1}}\cdots$$
 with $d_X^{i}d_X^{i+1}=0$ for all $i\in\mathbb{Z}$.
 We call $\cpx{X}$ a \emph{radical} complex if all $d_X^i$ are radical morphisms.
 We denote by $\C{\cal C}$ (respectively $\mathscr{C}^{\mathrm{b}}(\cal C)$)
 the category of complexes (respectively bounded complexes)
 over ${\cal C}$, and by $\K{\cal C}$ (respectively $\Kb{\cal C}$)
 the corresponding homotopy category.
 $\D{\cal C}$ (respectively $\Db{\cal C}$) is
 the derived category of complexes (respectively bounded complexes)
 over ${\cal C}$
 when ${\cal C}$ is abelian.
 Homotopy categories and derived categories are prominent examples
 of \emph{triangulated categories}.
%
%

 For an algebra $A$, we write $\mathscr{C}^*(A)$,
 $\mathscr{K}^*(A)$ and $\mathscr{D}^*(A)$ for
 $\mathscr{C}^*(\modcat{A})$, $\mathscr{K}^*(\modcat{A})$ and
 $\mathscr{D}^*(\modcat{A})$ respectively, where $*$ stands for $\textrm{blank}$
 or $\mathrm{b}$.
 \begin{Lem} \label{lem:morphisms in derived categories}
   Let $A$ be an algebra. Then:

   $(1)$ Every complex of $A$-modules is isomorphic to a radical complex in the homotopy category $\K{A}$.

  $(2)$ Two radical complexes are isomorphic in the homotopy category $\K{A}$ if and only if so they are in $\C{A}$.

   $(3)$ For two complexes
   $\cpx{X}$ and $\cpx{Y}$, if there exists an integer $n$ such that $\cpx{X}$ has
   no cohomology in degrees larger than $n$ (i.e., $\mathrm{H}^i(\cpx{X})=0$
   for $i>n$), and $\cpx{Y}$ has no cohomology in degrees smaller
   than $n$ (i.e., $\mathrm{H}^i(\cpx{Y})=0$ for $i<n$), then
   $$\Hom_{\D{A}}(\cpx{X},\cpx{Y})\cong \Hom_A(\mathrm{H}^n(\cpx{X}),\mathrm{H}^n(\cpx{Y})).$$ \\
   In particular, for any complex $\cpx{Z}$ of $A$-modules,
   $\Hom_{\D{A}}(A, \cpx{Z}[i])$ is isomorphic to its $i$-th cohomology $\mathrm{H}^i(\cpx{Z})$
   for all $i$.
 \end{Lem}

 \begin{proof}
   The first two statements are taken from \cite[p. 112-113]{HuXi10};
   the remaining statements can be shown by using
 truncations and Cartan-Eilenberg resolutions of complexes.
 \end{proof}

 The tensor product of two complexes $\cpx{X}$ and $\cpx{Y}$ in $\C{\Modcat{A^{\mathrm{e}}}}$
 is defined to be the total complex of the double complex with its $(i,j)$-term $X^i\otimes_A Y^j$, and
 their tensor product in $\D{\Modcat{A^{\mathrm{e}}}}$ is the tensor
 product of their projective resolutions in $\C{\Modcat{A^{\mathrm{e}}}}$ \cite{Sp85}.


 Let $\cpx{X}$ be a complex over ${\cal C}$  of the  form
 $\cdots\longrightarrow X^{i-1}\xrightarrow{d_X^{i-1}}X^i\xrightarrow{d_X^i}X^{i+1}\xrightarrow{d_X^{i+1}}\cdots$. Then the {\em brutal truncations} of
 $\cpx{X}$ are defined by cutting off the left or right hand part of the
 complex: $\cpx{X}_{\geq i} = \sigma_{\geq i}(\cpx{X}): \cdots\longrightarrow 0
 \longrightarrow X^i \xrightarrow{d_X^i} X^{i+1}\xrightarrow {d_X^{i+1}}\cdots$
 and  $\cpx{X}_{< i} = \sigma_{< i}(\cpx{X}):
 \cdots\longrightarrow X^{i-2}\xrightarrow{d_X^{i-2}}
 X^{i-1}\longrightarrow 0 \longrightarrow \cdots$. There is an exact
 sequence of complexes $0 \rightarrow \sigma_{\geq i}(\cpx{X}) \rightarrow
 \cpx{X} \rightarrow \sigma_{< i}(\cpx{X}) \rightarrow 0$, which also defines
 a triangle.

 \subsection{Derived equivalences and tilting complexes}

 {\em Derived equivalences} are by definition equivalences of derived categories
 that preserve the triangulated structures, that is shift and
 triangles. Two algebras \emph{$A$ and $B$ are derived
   equivalent} if there is a derived equivalence between their derived
 categories. Despite their importance, derived equivalences are still rather
 unknown and even basic questions are still open.
 The equivalence relation between algebras defined by derived equivalence
 does, however, admit a very satisfactory theory, known as \emph{Morita theory
   for derived categories}, due to Rickard and (more generally for dg algebras)
 to Keller.

 \begin{Theo}[(Rickard \cite{R89}, Keller \cite{Ke94})] 
\label{thm:morita theory for derived category}
   Let $A$ and $B$ be two $k$-algebras. The following statements are equivalent.
   \begin{itemize}
    \setlength\itemsep{-1pt}
     \item[\rm (1)]  $\D{\Modcat{A}}$ and $\D{\Modcat{B}}$ are equivalent as triangulated categories.
     \item[\rm (2)] $\Db{\Modcat{A}}$ and $\Db{\Modcat{B}}$ are equivalent as triangulated categories.
     \item[\rm (3)] $\Db{A}$ and $\Db{B}$ are equivalent as triangulated categories.
     \item[\rm (4)] $\Kb{\pmodcat{A}}$ and $\Kb{\pmodcat{B}}$ are equivalent as triangulated categories.
     \item[\rm (5)] There exists a complex $\cpx{T}\in \Kb{\pmodcat{A}}$ such that
       $\End_{\D{A}}(\cpx{T})\cong B$
   and
     \begin{itemize}
       \setlength\itemsep{-1pt}
       \item[\rm (a)] $\Hom_{\D{A}}(\cpx{T},\cpx{T}[n])=0$ unless $n=0$;
       \item[\rm (b)] $\add(\cpx{T})$ generates $\Kb{\pmodcat{A}}$ as triangulated category.
     \end{itemize}
   \end{itemize}
 \end{Theo}

 The complex  $\cpx{T}$ in (5)  is called a \emph{tilting complex}.

 For any derived equivalence from $\D{\Modcat{A}}$ to $\D{\Modcat{B}}$, the
 image of $A$ in $\D{\Modcat{B}}$ is a tilting complex, and so is the preimage of $B$ in $\D{\Modcat{A}}$.
 It is not known whether the equivalences in (1)-(4)  determine each other
 uniquely (see  \cite[Section 7]{R89}). To fix the ambiguity,
 Rickard \cite{R91} associated to each derived equivalence a \emph{standard derived
 equivalence}.  A complex $\cpx{\Delta}\in \Db{\Modcat{(B\otimes_kA\opp)}}$ is called a {\em two-sided tilting complex} if and only if
 $$\cpx{\Delta}\otimesL_A\cpx{\Theta}\cong {}_BB_B\quad \mbox{and}\quad\cpx{\Theta}\otimesL_B\cpx{\Delta}\cong {}_AA_A$$ for some complex $\cpx{\Theta}\in\Db{\Modcat{(A\otimes_kB\opp)}}$.  The complex $\cpx{\Theta}$ is called an 
{\em inverse} of $\cpx{\Delta}$. 
The functor $\cpx{\Delta}\otimesL_A-: \D{\Modcat{A}}\ra\D{\Modcat{B}}$ 
(respectively, $-\otimesL_B\cpx{\Delta}: 
\D{\Modcat{B\opp}}\ra\D{\Modcat{A\opp}}$) 
is a triangle equivalence with $\cpx{\Theta}\otimesL_B-$ (respectively, $-\otimesL_A\cpx{\Theta}$) as a quasi-inverse.  Such a derived equivalence is called a {\em standard derived equivalence}.
 It has been proved in \cite{R91} that each derived equivalence $\mathcal{F}: \Db{\Modcat{A}}\cong \Db{\Modcat{B}}$ induces a derived equivalence $\hat{\mathcal{F}}$ from $\Db{\Modcat{(B\otimes_k B\opp)}}$ to $\Db{\Modcat{(B\otimes_k A\opp)}}$, and the image $\cpx{\Delta}$ of ${}_BB_B$ under $\hat{\mathcal{F}}$ is a two-sided tilting complex such that $\cpx{\Delta}\otimesL_A\cpx{X}\cong \mathcal{F}(\cpx{X})$ for all $\cpx{X}\in \Db{\Modcat{A}}$.  It is not known whether $\mathcal{F}$ and $\cpx{\Delta}\otimesL_A-$ agree on morphisms.

 Two-sided tilting complexes can be used to provide further derived equivalences.  Let $\cpx{\Delta}$ be a two-sided tilting complex in $\Db{\Modcat{(B\otimes A\opp)}}$ with an inverse $\cpx{\Theta}$.  Then, for each algebra $C$, the functor $$\cpx{\Delta}\otimesL_A-: \D{\Modcat{(A\otimes_kC\opp)}}\ra\D{\Modcat{(B\otimes_kC\opp)}}$$ is a triangle equivalence with $\cpx{\Theta}\otimesL_B-$ as a quasi-inverse, and the functor
 \begin{align} \label{eqn:two sided standard derived equivalence}
    \check{\mathcal{F}}:=\cpx{\Delta}\otimesL_A -\otimesL_A \cpx{\Theta}: \D{\Modcat{A^\mathrm{e}}}
    \lra \D{\Modcat{B^{\mathrm{e}}}}
 \end{align}
 defines a derived equivalence with $\cpx{\Theta}\otimesL_B-\otimesL_B\cpx{\Delta}$ as a quasi-inverse, see \cite{R91} for more details.

 \vspace{-.2cm}
\subsection{Standard derived equivalences} \label{subsectionstandarddereq}

 Here are
 some properties of the standard derived equivalence $\check{\mathcal{F}}$,
 to be used later on.

 \begin{Lem} \label{lem:standard derived equivalence}
 Let $A$ and $B$ be derived equivalent $k$-algebras. Let $\check{\mathcal{F}}$ be a
 standard derived equivalence from $\D{\Modcat{A^{\mathrm{e}}}}$ to $\D{\Modcat{B^{\mathrm{e}}}}$
 defined by a two-sided tilting complex $\cpx{\Delta}$ in $\D{B\otimes A\opp}$ as defined above
 in (\ref{eqn:two sided standard derived equivalence}). Then
 for any complexes $\cpx{X}$ and $\cpx{Y}$ in $\D{\Modcat{A^{\mathrm{e}}}}$,
 there are isomorphisms in $\D{\Modcat{B^{\mathrm{e}}}}$
 \begin{itemize}
  \setlength\itemsep{-1pt}
   \item[\rm (1)] $\check{\mathcal{F}}(A)\cong B$, $\check{\mathcal{F}}(\DD(A))\cong \DD(B)$;
   \item[\rm (2)] $\check{\mathcal{F}}(\cpx{X}\otimesL_A\cpx{Y})\cong \check{\mathcal{F}}(\cpx{X})\otimesL_B\check{\mathcal{F}}(\cpx{Y})$;
   \item[\rm (3)] $\check{\mathcal{F}}\big(\RHom_A({_A}\cpx{X}, {_A}\cpx{Y})\big)\cong
   \RHom_B({_B}\check{\mathcal{F}}(\cpx{X}), {_B}\check{\mathcal{F}}(\cpx{Y}))$;
   \item[\rm (4)] $\check{\mathcal{F}}(\DD(\cpx{X})) \cong \DD(\check{\mathcal{F}}(\cpx{X}))$.
 \end{itemize}
\end{Lem}

\begin{proof}
 (1) follows from the isomorphism $\cpx{\Delta}\otimesL_AA\otimesL_A\cpx{\Theta}\cong B$ and  (2) follows from \cite[Proposition 5.2]{R91}.
 To show (3) and (4), note that the derived functor $\RHom_{B}(\cpx{\Delta},-)$
 from $\D{\Modcat{(B\otimes A\opp)}}$ to $\D{\Modcat{A^{\mathrm e}}}$ is right adjoint to
 the derived functor
 $\cpx{\Delta}\otimesL_A-$ (see \cite{Mi03,Sp85}). Thus it is naturally isomorphic to
 the derived functor $\cpx{\Theta}\otimesL_B-$. Similarly, the two derived functors $\RHom_{A}(\cpx{\Theta},-)$ and $\cpx{\Delta}\otimesL_A-$
 are naturally isomorphic. (3) then follows by a series of isomorphisms in $\D{\Modcat{B^{\mathrm e}}}$:
 \begin{align*}
   \RHom_{B}({_B}\check{\mathcal{F}}(\cpx{X}), {_B}\check{\mathcal{F}}(\cpx{Y})) &
   = \RHom_{B}(\cpx{\Delta}\otimesL_A\cpx{X}\otimesL_A\cpx{\Theta}, \cpx{\Delta}\otimesL_A\cpx{Y}\otimesL_A\cpx{\Theta})\\
   & \stackrel{(\ast)}{\cong} \RHom_{A}(\cpx{X}\otimesL_A\cpx{\Theta}, \RHom_{B} (\cpx{\Delta}, \cpx{\Delta}\otimesL_A\cpx{Y}\otimesL_A\cpx{\Theta}))
   \\
   & \cong \RHom_{A}(\cpx{X}\otimesL_A\cpx{\Theta}, \cpx{\Theta}\otimesL_B\cpx{\Delta}\otimesL_A\cpx{Y}\otimesL_A\cpx{\Theta})
   \\
   & \cong \RHom_{A}(\cpx{X}\otimesL_A\cpx{\Theta}, \cpx{Y}\otimesL_A\cpx{\Theta})\\
   & \stackrel{(\ast)}{\cong} \RHom_{A}(\cpx{\Theta},  \RHom_{A}(\cpx{X}, \cpx{Y}\otimesL_A\cpx{\Theta}))\\
   & \cong \cpx{\Delta}\otimesL_A\RHom_{A}(\cpx{X}, \cpx{Y}\otimesL_A\cpx{\Theta})\\
   & \stackrel{(\dagger)}{\cong} \cpx{\Delta}\otimesL_A\RHom_{A}({_A}\cpx{X}, {_A}\cpx{Y})\otimesL_A\cpx{\Theta} \\
   & = \check{\mathcal{F}}(\RHom_{A}({_A}\cpx{X}, {_A}\cpx{Y})).
 \end{align*}
 Here the isomorphisms marked by $(\ast)$ follow by tensor-hom adjointness, and
 the isomorphism marked by $(\dagger)$
 follows from ${_A}\cpx{\Theta}\in \Kb{\pmodcat{A}}$.

 To prove (4), observe that $\DD(\cpx{X})\cong \RHom_{A}(\cpx{X}, \DD(A))$
 in $\D{\Modcat{A^\mathrm{e}}}$. Thus by (1) and (3)
 \begin{align*}
    \check{\mathcal{F}}(\DD(\cpx{X})) \cong \check{\mathcal{F}}(\RHom_{A}({_A}\cpx{X},{_A}\DD(A)))
    &\cong \RHom_{B}({_B}\check{\mathcal{F}}(X),{_B}\check{\mathcal{F}}(\DD(A))) \\
    &\cong \RHom_{B}({_B}\check{\mathcal{F}}(X),{_B}\DD(B))\cong \DD(\check{\mathcal{F}}(\cpx{X})),
 \end{align*}
 in $\D{\Modcat{B^{\mathrm{e}}}}$.
 \end{proof}

\begin{Lem}
Let ${\mathcal{T}}$ be a triangulated category, and let $\xi_i: X_i\ra Y\raf{f_i} Z\ra X_i[1], i=1, 2$ be triangles in ${\mathcal{T}}$.  If one of the following conditions is satisfied

$(1)$  $\Hom_{\mathcal{T}}(Y, X_i)=0=\Hom_{\mathcal{T}}(Y, X_i[1])$ for $i=1, 2$;

$(2)$ $\Hom_{\mathcal{T}}(X_i, Z)=0=\Hom_{\mathcal{T}}(X_i[1], Z)$ for $i=1, 2$,

\noindent
then $\xi_1$ and $\xi_2$ are isomorphic.
\label{lemma-isomorphic-triangles}
\end{Lem}
\begin{proof}
  Assume that condition (1) is satisfied. Applying $\Hom_{\mathcal{T}}(Y, -)$ to $\xi_i$ ($i=1, 2$) yields isomorphisms $\Hom_{\mathcal{T}}(Y, f_i): \Hom_{\mathcal{T}}(Y, Y)\ra \Hom_{\mathcal{T}}(Y, Z)$. This means that each morphism 
\mbox{$g: Y\ra Z$} factorises uniquely through $f_i$.
  In particular, $f_1=hf_2$  and $f_2=h'f_1$ for some 
\mbox{$h, h'\in\End_{\mathcal{T}}(Y)$}. 
It follows that $f_1=hh'f_1$ and $f_2=h'hf_2$, and thus by uniqueness,
 $hh'=1_{Y}=h'h$.  Hence $h: Y\ra Y$ is an isomorphism, and the commutative diagram
$$\xymatrix{
Y\ar[r]^{f_1}\ar[d]_{h} & Z\ar@{=}[d]\\
Y\ar[r]^{f_2}  & Z
}$$
extends to an isomorphism between $\xi_1$ and $\xi_2$.
The proof is similar when assuming (2).
\end{proof}

 Derived equivalences, by definition, preserve all triangles,
 especially the following type. Let $e=e^2$ be an idempotent in
 the $k$-algebra $A$. There are canonical
 triangles associated to $e$ in $\D{\Modcat{A^{\mathrm e}}}$:
 \begin{align*}
\mathbf{(CT1)} \, \, \,  \xi_e^A& :\quad  U_A(e)\lra Ae\otimesL_{eAe}eA\lraf{\pi_e}A\lra U_A(e)[1] \\
\mathbf{(CT2)} \, \, \,  \eta_e^A &:\quad  A \lraf{\rho_e} \RHom_{eAe}(eA,eA) \lra V_A(e) \lra A[1]
 \end{align*}
 where $\pi_e: Ae\otimesL_{eAe} eA\to A$ is induced by the multiplication map
 $Ae\otimes_{eAe} eA\to A$, and $\rho_e$ is induced by the canonical morphism
 $A\to \End_{eAe}(eA)$.
 The triangle $\xi_e^A$ plays a crucial role in the analysis of recollements of
 derived categories \cite{Mi03}.
 The following result implies that the property of being a canonical
 triangle is preserved under certain derived equivalences.

\begin{Lem}\label{lem:derived-equivalence-preserves-canonical-triangle}
 Let $A$ and $B$ be derived equivalent $k$-algebras, and ${}_B\cpx{\Delta}_A$
 a two-sided tilting complex with inverse ${}_A\cpx{\Theta}_B$.
 Let $e$ and $f$ be idempotents in $A$ and $B$ respectively.
 Assume that the standard derived equivalence $\cpx{\Delta}\otimesL_A-:
 \D{\Modcat{A}}\ra \D{\Modcat{B}}$ restricts to a triangle equivalence
 $\Kb{\add({}_AAe)}\cong \Kb{\add({}_BBf)}$.
 Then
 $\cpx{\Delta}\otimesL_A\xi_e^A\otimesL_A\cpx{\Theta}\cong \xi_f^B$
 and $\cpx{\Delta}\otimesL_A \eta_e^A\otimesL_A\cpx{\Theta}\cong \eta_f^B$
 as triangles in $\D{\Modcat{B^{\mathrm e}}}$.
\end{Lem}

\begin{proof}
 Since $eA\otimes_A-: \add{({_A}Ae)}\to\pmodcat{eAe}$ is an equivalence of additive categories
 with a quasi-inverse $Ae\otimes_{eAe}-: \pmodcat{eAe}\to \add{({_A}Ae)}$,
 it follows that $eA\otimesL_A-: \Kb{\add({}_AAe)}\to \Kb{\pmodcat{eAe}}$ is an equivalence
 of triangulated categories with a quasi-inverse $Ae\otimesL_{eAe}- $.
 Similarly, the derived functor $fB\otimesL_B-$ defines an equivalence
 from $\Kb{\add({}_{B}Bf)}$
 to $\Kb{\pmodcat{fBf}}$ with a quasi-inverse $Bf\otimesL_{fBf}-$.
 Note that $\cpx{\Theta}\cong \RHom_B(\cpx{\Delta}, B)$ and 
\mbox{$\cpx{\Delta}\cong \RHom_A(\cpx{\Theta}, A)$} (see \cite{R91}).

\smallskip
 {\it Claim.}

$\begin{array}{lll}
\phantom{xxxxx} &(\mathrm a) \, \cpx{\Delta}e\in \Kb{\add({}_BBf)},  &\qquad \cpx{\Theta}f \in \Kb{\add({}_AAe)}, \\
\phantom{xxxxx} & (\mathrm b) \, f\cpx{\Delta}\in \Kb{\add(eA_A)}, & \qquad e\cpx{\Theta}\in \Kb{\add(fB_B)}, \\
 \phantom{xxxxx} &(\mathrm c) \, f\cpx{\Delta}e\otimesL_{eAe}e\cpx{\Theta}f\cong fBf,  &\qquad e\cpx{\Theta}f\otimesL_{fBf}f\cpx{\Delta}e\cong eAe.
\end{array}$

\medskip
 \noindent{\em Proof of claim.}
 By assumption, $\cpx{\Delta}e\cong \cpx{\Delta}\otimesL_AAe$ belongs to $\Kb{\add({}_BBf)}$. Similarly \mbox{$\cpx{\Theta}f\in \Kb{\add({}_AAe)}$}, hence (a).

 There are the following isomorphisms
 $$f\cpx{\Delta}\cong \RHom_B(Bf, \cpx{\Delta})\cong \RHom_B(Bf, \RHom_A(\cpx{\Theta}, A))\cong\RHom_A(\cpx{\Theta}f, A).$$
 It follows that $f\cpx{\Delta}\in\Kb{\add(eA_A)}$.  By symmetry, also $e\cpx{\Theta}\in \Kb{\add(fB_B)}$, hence (b).

 To prove the first statement in (c), consider the isomorphisms
 $$f\cpx{\Delta}e\otimesL_{eAe}e\cpx{\Theta}f\cong f\cpx{\Delta}\otimesL_AAe
 \otimesL_{eAe}eA\otimesL_A\cpx{\Theta}f\cong fB\otimesL_B\cpx{\Delta}\otimesL_A
 \cpx{\Theta}\otimesL_BBf\cong fBf$$
 which use
 $eA\otimesL_A-: \Kb{\add({}_AAe)}\to \Kb{\pmodcat{eAe}}$ being an equivalence
 with quasi-inverse $Ae\otimesL_{eAe}- $ and
 $\cpx{\Delta}\otimesL_A\cpx{\Theta}\cong {}_BB_B$.
 The second statement in (c) follows similarly.
\bigskip

 Now, the isomorphisms (the second one using (a) and (c))
 \begin{align*}
   \cpx{\Delta}\otimesL_A Ae\otimesL_{eAe}eA\otimesL_A \cpx{\Theta} & \cong\cpx{\Delta}e\otimesL_{eAe}e\cpx{\Theta} \\   &  \cong Bf\otimesL_{fBf}f\cpx{\Delta}e\otimesL_{eAe} e\cpx{\Theta}f\otimesL_{fBf}fB\\ & \cong Bf\otimesL_{fBf}fB
 \end{align*}
 combined with Lemma \ref{lem:standard derived equivalence} (1) yield
 that $\cpx{\Delta}\otimesL_A \xi_e^A \otimesL_A \cpx{\Theta}$ is a
 triangle of the following form
 \[
    \delta_f^B :\qquad U_B(f)' \lra Bf\otimesL_{fBf} fB \lraf{\epsilon} B \lra U_B(f)'[1]
 \]
 in $\D{\Modcat{B^{\mathrm{e}}}}$, where $U_B(f)' = \cpx{\Delta}\otimesL_A U_A(e)\otimesL_A\cpx{\Theta}$. So, the triangles $\xi_f^B$ and $\delta_f^B$ have at
 least two terms in common. To identify them as triangles, note
 that $eU_A(e)=0$ and
 \mbox{$fU_B(f)' = f\cpx{\Delta}\otimesL_AU_A(e)\otimesL_A\cpx{\Theta}=0$} 
since $f\cpx{\Delta}$
 belongs to $\Kb{\add(eA_A)}$ by (b).
 Then
 $$\Hom_{\D{\Modcat{B^{\mathrm{e}}}}}(Bf\otimesL_{fBf}fB, U_B(f)[i])=0=
 \Hom_{\D{\Modcat{B^{\mathrm{e}}}}}(Bf\otimesL_{fBf}fB, U_B(f)'[i]),   \forall i\in \mathbb{Z}.
 $$
 Here, the vanishing follows by adjointness.
 Thus $\xi_f^B$ and $\delta_f^B$ are isomorphic
 triangles in $\D{\Modcat{B^{\mathrm{e}}}}$
 by Lemma \ref{lemma-isomorphic-triangles}. So,  $\xi_f^B$ is as claimed.

\medskip
 It remains to check the claim about $\eta_f^B$.

\medskip
 {\it Claim.} The canonical triangle $\eta_e^A$ is isomorphic to
 $\RHom_A(\xi_e^A, A)$.

\medskip
\noindent {\it Proof of claim.} Applying $\RHom_A(-, A)$ to $\xi_e^A$
 results in a triangle
 $$\eta': A\lra \RHom_A(Ae\otimesL_{eAe}eA, A)\lra \RHom_A(U_A(e), A)\lra A[1], $$
 and $\RHom_A(Ae\otimesL_{eAe}eA, A)\cong \RHom_{eAe}(eA, eA)$ by adjointness.
 Let $V':=\RHom_A(U(e), A)$. By definition, $U_A(e)e=0$. Using
 adjointness again, this implies $\RHom_A(Ae, V')=0$,  that is,  $eV'=0$.
 Similarly, $eV_A(e)=0$. It follows that
 $$\Hom_{\D{\Modcat{A^{\rm e}}}}(V'[i], \RHom_{eAe}(eA, eA))=0=\Hom_{\D{\Modcat{A^{\rm e}}}}(V_A(e)[i], \RHom_{eAe}(eA, eA))$$
 for all $i\in\mathbb{Z}$.
 By Lemma \ref{lemma-isomorphic-triangles}, the triangles $\eta'$ and $\eta_e^A$ are isomorphic, which proves the claim.

\medskip
 By the first part of the proof, $\cpx{\Delta}\otimesL_A\xi_e^A\otimesL_A\cpx{\Theta}\cong \xi_f^B$.
 Applying Lemma \ref{lem:standard derived equivalence} (3) shows that
${\cpx{\Delta}\otimesL_A\eta_e^A\otimesL_A\cpx{\Theta}\cong \eta_f^B}$.
 \end{proof}

\subsection{Almost $\nu$-stable derived equivalences}

Derived equivalences in general fail to preserve homological invariants
such as global or dominant dimension. In this respect, stable equivalences
of Morita type behave much better. Unfortunately, derived equivalences between
algebras that are not self-injective, in general do not induce stable
equivalences. The problem of finding derived equivalences, which do induce
stable equivalences of Morita type,
has been addressed in \cite{HuXi10} by introducing
a new class of derived equivalences, called \emph{almost $\nu$-stable
 derived equivalences}, and relating them with
 stable equivalences.
 As a crucial feature, these derived equivalences
 preserve many homological invariants.

\begin{Def}[(\cite{HuXi10})]
  Let $\mathcal{F}: \Db{A}\to \Db{B}$ be a derived equivalence between
  two $k$-algebras $A$ and $B$. We call $\mathcal{F}$
  an {\em almost $\nu$-stable derived equivalence} if the following conditions
  are satisfied.

$(1)$ The radical tilting complex $\mathcal{F}(A) = (\overline{T}^i, \overline{d}^i)_{i\in \mathbb{Z}}$
    in $\Kb{\pmodcat{B}}$ has nonzero terms only in positive degrees, that is, $\overline{T}^i=0$ for all $i<0$;
    the radical tilting complex $\mathcal{F}^{-1}(B) = (T^i, d^i)_{i\in \mathbb{Z}}$ in
    $\Kb{\pmodcat{A}}$ has nonzero terms only in negative degrees, that is, $T^i=0$ for all $i>0$.

$(2)$ ${\add(\bigoplus_{i<0}T^{i})=\add(\bigoplus_{i<0}\nu_{_A}T^{i})}$
    and ${\add(\bigoplus_{i>0}\overline{T}^i)=\add(\bigoplus_{i>0}\nu_{_B}\overline{T}^i)}$.
\end{Def}
 Note that we may assume without loss of generality that the tilting complex $\mathcal{F}(A)$ is radical
 by Lemma \ref{lem:morphisms in derived categories}. The two conditions for $\cpx{T}$
 are equivalent to those for $\cpx{\overline{T}}$ \cite{HuXi10}. To generalise
 this type of derived equivalence, but to keep many
 interesting properties,
 the following \emph{iterated almost $\nu$-stable derived equivalences} have
 been introduced.

\begin{Def}[(\cite{Hu12})]
  Let $\mathcal{F}: \Db{A}\to \Db{B}$ be a derived equivalence between
  two $k$-algebras $A$ and $B$.
  We call $\mathcal{F}$ an
  {\em iterated almost $\nu$-stable derived equivalence}, if
  there exists a sequence of derived equivalences $\mathcal{F}_i: \D{A_i}\cong \D{A_{i+1}}$
  of $k$-algebras $A_i$
  $(0\leq i\leq N)$ for some $N\in \mathbb{N}$ with $A_0=A$ and $A_{N+1}=B$ such that
  each $\mathcal{F}_i$ or $\mathcal{F}_i^{-1}$ is an almost
  $\nu$-stable derived equivalence and $\mathcal{F} \cong \mathcal{F}_N\circ \cdots \circ \mathcal{F}_0$.
\end{Def}

In Section \ref{sec:derived equivalences for morita algebras} we will
see that
 all derived equivalences between certain classes of algebras are
 iterated almost $\nu$-stable derived equivalences. This will make full
 use of the characterisations developed in \cite{Hu12}.
The crucial property in our context is:

\begin{Prop}[(\cite{Hu12})]
  \label{prop:almost nu-stable preserve dimensions}
   Let $\mathcal{F}: \Db{A}\to \Db{B}$ be an iterated almost $\nu$-stable
   derived equivalences between $k$-algebras.
   Then $\gldim(A)=\gldim(B)$ and $\domdim(A)=\domdim(B)$.
 \end{Prop}

 \section{Derived restriction theorem - from Morita algebras to
 self-injective algebras}
\label{sec:derived equivalences for morita algebras}

The derived restriction theorem, to be proved in this section, states that
derived equivalences between two algebras restrict to derived equivalences
between their associated self-injective centraliser subalgebras, provided the
two given algebras have $\nu$-dominant dimension at least one.
A subcategory of the bounded homotopy
category of projective modules will be defined and shown to be invariant under
derived equivalence, under the assumption on $\nu$-dominant dimension. This
will be the key ingredient in the proof of the derived restriction theorem.

Recall that for an algebra $A$ with $\Ndomdim(A)\geq 1$ and minimal
faithful left module $Ae$, the algebra $H=H_A=eAe$ is called the
 \emph{associated self-injective algebra of $A$}. Then the category $\stp{A}$
is additively generated by $Ae$. That $H$ is self-injective has been shown
in the proof of Lemma \ref{lem:two dominant dimensions coincide}.

 Suppose now that $A$ and $B$ are derived equivalent
 $k$-algebras. Like dominant dimension, $\nu$-dominant dimension is
 not a derived invariant. In particular, $\Ndomdim(A)\geq 1$ and
$\Ndomdim(B) = 0$ may happen, and there may be no reasonable way to define
a non-zero
associated self-injective algebra for $B$ in this case. An example of such
a situation has been given, for different reasons,
in \cite[Section 5]{Rickardderivedstable}.
Therefore,
it seems difficult to deduce any connection between the
 associated self-injective algebras of $A$ and $B$ in general without
assuming both algebras have $\nu$-dominant dimension at least one.

If we, however, assume that both $A$ and $B$ have
 $\nu$-dominant dimension at least $1$, then
Theorem \ref{thm:restriction theorem}
 below shows that any derived equivalence between $A$ and $B$ restricts to a
 derived equivalence between their associated self-injective algebras.
The main tool for
 showing is the following subcategory of the homotopy category.

 \begin{Def}\label{defn:nu-stable complex}
   Let $A$ be a $k$-algebra and $\nu_{_A}$ be the Nakayama functor.
   Define
   $$\nuinv{A}:=\big\{\cpx{P}\in\Kb{\pmodcat{A}}|\cpx{P}\cong \nu_{_A}(\cpx{P})\mbox{ in }\Db{A}\big\}.$$
 \end{Def}

   Note that for $\cpx{P}\in \Kb{\pmodcat{A}}$, the complex $\nu_{_A}(\cpx{P})$ is defined componentwise. When $A$ has arbitrary $\nu$-dominant dimension,
 a complex in $\nuinv{A}$ need not be isomorphic in $\K{\modcat{A}}$
 to a complex in $\Kb{\stp{A}}$; this is illustrated by an example below.
 However:

 \begin{Prop}\label{prop:nu-stable complexes}
 Let $A$ be a $k$-algebra with $\Ndomdim(A)\geq 1$.
 Then $\Kb{\stp{A}}$ is the smallest triangulated full subcategory of
 $\Kb{\pmodcat{A}}$ that contains $\nuinv{A}$ and is closed under taking direct summands.
 In particular,
 every complex in $\nuinv{A}$ is isomorphic in $\Kb{\pmodcat{A}}$ to a complex in the
 category $\Kb{\stp{A}}$.
 \end{Prop}

 \begin{proof}
   Let $\thick(\nuinv{A})$ be the smallest triangulated full subcategory of
   $\Kb{\pmodcat{A}}$ which contains $\nuinv{A}$ and is closed under taking direct
   summands. Let $E$ be a basic additive generator of $\stp{A}$.
   Then $\nu_{_A}(E)$ is isomorphic to $E$ in $\modcat{A}$, because
   $\nu_{_A}(E)$ is again basic and strongly projective-injective and has the
   same number of indecomposable direct summands as $E$.
   Therefore $E$ belongs to $\nuinv{A}$. Since $\Kb{\stp{A}}$ is the smallest triangulated
   full subcategory of $\Kb{\pmodcat{A}}$ which contains $E$ and is closed under
   taking direct summands, it follows that $\Kb{\stp{A}}\subseteq \thick{(\nuinv{A})}$.

   To finish the proof, we need to show $\nuinv{A}\subseteq \Kb{\stp{A}}$,
   that is, every radical complex 
\mbox{$\cpx{P} = (P^i, d^i)_{i\in \mathbb{Z}}$} in $\nuinv{A}$
   is isomorphic in $\Kb{\pmodcat{A}}$ to a complex in $\Kb{\stp{A}}$.
   Without loss of generality, we assume that $\inf\{l\mid P^l\neq 0\}=0$.
   Let $n=\sup\{r\mid P^r\neq 0\}$. So the complex $\cpx{P}$ is of the form
   \[
   \cdots\lra 0\lra P^0\lraf{d^0} P^1\lraf{d^1} \cdots\to P^{n-1}\lraf{d^{n-1}} P^n\lra 0\lra \cdots.
   \]

   We will prove by induction on $n$ that $\cpx{P}$ is isomorphic in $\Kb{\pmodcat{A}}$ to a complex
   in $\Kb{\stp{A}}$.
   If $n=0$, then $\nu_{_A}(\cpx{P})\cong \cpx{P}$ in $\Db{A}$ implies that
   $\nu_{_A}(P^0)\cong P^0$ in $\modcat{A}$, and so $\cpx{P}\in \Kb{\stp{A}}$.
   In general, we first prove:

\smallskip
  {\it Claim.} $P^0$ is strongly projective-injective.

\smallskip
  {\it Proof.} Since $\cpx{P}$ is a radical complex in $\Kb{\pmodcat{A}}$
   and $\nu_{_A}: \pmodcat{A}\to \imodcat{A}$ is an equivalence, it follows
   that $\nu_{_A}(\cpx{P})$ is a radical complex in $\Kb{\imodcat{A}}$.
   Let $\cpx{f}=\{f^i\}: \cpx{P}\to \nu_{_A}(\cpx{P})$ be a quasi-isomorphism.
   Then $\nu_{_A}(\cpx{P})$ is an injective resolution of $\cpx{P}$.
   By the construction of Cartan-Eilenberg injective resolutions
   $\cpx{P}$ admits an injective resolution $\cpx{I}$ with the properties:
   $\cpx{I}$ is quasi-isomorphic to $\cpx{P}$, and
   $I^i=0$ for $i<0$ and $I^0$ is the injective envelope of $P^0$.
   By the uniqueness of injective resolutions up to homotopy,
   the radical complex $\nu_{_A}(\cpx{P})$ and the complex $\cpx{I}$
   are isomorphic in $\K{\imodcat{A}}$,
   and therefore $\nu_{_A}(P^0)$ is a direct summand of $I^0$ by Lemma
   \ref{lem:morphisms in derived categories}.
   Since $\Ndomdim(A)\geq 1$, the injective envelope $I^0$ of $P^0$ is
   strongly projective-injective. It follows that $\nu_{_A}(P^0)$ and
   hence $P^0$ are strongly projective-injective as well.

\smallskip
   {\it Claim. $f^0: P^0\to \nu_{_A}(P^0)$ is an
      isomorphism of $A$-modules. }

\smallskip
{\it Proof.}
   Let $\mathrm{Cone}(\cpx{f})$ be the mapping cone of $\cpx{f}$, a complex of the form
   \[
       \mathrm{Cone}(\cpx{f}):\qquad 0\lra P^0 \lraf{[-d^0,f^0]} P^1\oplus \nu_{_A}(P^0)\lra\cdots  \lra \nu_{_A}(P^n)\lra 0
   \]
   where $P^{0}$ is placed in degree $-1$. Since $\cpx{f}$ is a quasi-isomorphism,
   it follows that
   $\mathrm{Cone}(\cpx{f})$ is an acyclic complex, and thus
   the morphism $[-d^0, f^0]: P^0\to P^1\oplus \nu_{_A}(P^0)$ splits in $\modcat{A}$
   because $P^0$ is injective.
   Let $u: P^1\lra P^0$ and $v: \nu_{_A}(P^0)\lra P^0$ be morphisms
   in $\modcat{A}$ such that $-d^0u+f^0v=1_{P^0}$.
   Then $f^0v = 1+d^0u$ is invertible in $\End_A({P^0})$
   by the assumption that $d^0$ is a radical morphism.
   As a result, $f^0$ is a split monomorphism, and even
   an isomorphism as
   $P^0$ and $\nu_{_A}(P^0)$ have the same number of indecomposable direct summands.

\smallskip
   Now, let $\cpx{P}_{\geq 1}:=\sigma_{\geq 1}(\cpx{P})$
   be the brutal truncation of $\cpx{P}$
   and let
   $\cpx{f}_{\geq 1}: \cpx{P}_{\geq 1}\lra \nu_{_A}(\cpx{P}_{\geq 1})$ be
   the corresponding truncation of the chain morphism $\cpx{f}$.
   Starting from the  triangle
   \mbox{$P^0[-1]\lra \cpx{P}_{\geq 1} \lra \cpx{P} \lra P^0$},
   the following commutative diagram in $\D{A}$ gives a morphism of
   triangles in $\D{A}$
   \[
   \xymatrix{
    P^0[-1]\ar[r]\ar[d]^{f^0[-1]} & \cpx{P}_{\geq 1}\ar[r]\ar[d]^{\cpx{f}_{\geq 1}} & \cpx{P}\ar[r]\ar[d]^{\cpx{f}} & P^0\ar[d]^{f^0}\\
    \nu_{_A}(P^0)[-1]\ar[r] & \nu_{_A}(\cpx{P}_{\geq 1}) \ar[r] & \nu_{_A}(\cpx{P})\ar[r] & \nu_{_A}(P^0). }
   \]
   Since $\cpx{f}$ is a quasi-isomorphism (by assumption) and so is $f^0$ (by the arguments above),
   it follows that $\cpx{f}_{\geq 1}$ is a quasi-isomorphism. Therefore
   $\cpx{P}_{\geq 1}\in \nuinv{A}$, and by induction $\cpx{P}_{\geq 1}[1]$ is isomorphic in $\K{A}$ to
   a complex in $\Kb{\stp{A}}$. Using $P^0\in \stp{A}$ implies
   that $\cpx{P}$ is isomorphic in $\K{A}$ to a complex in $\Kb{\stp{A}}$.
 \end{proof}

 \noindent \textit{Example.} Without the
 assumption on $\nu$-dominant dimension, Proposition \ref{prop:nu-stable complexes}
 may fail in general:
 Let $A$ be the $k$-algebra given by the quiver
 $$\xymatrix{
  1\ar@/^/[r]^{\alpha} & 2\ar@/^/[l]^{\beta}\ar@/^/[r]^{\gamma} &3\ar@/^/[l]^{\delta} }
 $$
 and relations $\{\beta\delta, \alpha\beta\alpha, \delta\gamma\delta, \alpha\beta-\delta\gamma\}$.
 The indecomposable projective left $A$-modules are
 \[
 \xymatrix@C=.5mm@R=.5mm{ \\ P_1= \\} \xymatrix@C=.5mm@R=.5mm{&1\\&2\\1&&3\\}\hspace{1cm}
 \xymatrix@C=.5mm@R=.5mm{ \\ P_2= \\} \xymatrix@C=.5mm@R=.5mm{&2\\1&&3\\&2\\&&3}\hspace{1cm}
 \xymatrix@C=.5mm@R=.5mm{ \\ P_3= \\} \xymatrix@C=.5mm@R=.5mm{3\\2\\3}
 \]
 The indecomposable injective left $A$-modules are
 \[
 \xymatrix@C=.5mm@R=.5mm{ \\ I_1= \\} \xymatrix@C=.5mm@R=.5mm{1\\2\\1}\hspace{1cm}
 \xymatrix@C=.5mm@R=.5mm{ \\ I_2= \\} \xymatrix@C=.5mm@R=.5mm{&2\\1&&3\\&2}\hspace{1cm}
 \xymatrix@C=.5mm@R=.5mm{ \\ I_3= \\} \xymatrix@C=.5mm@R=.5mm{&2\\1&&3\\&2\\&&3}
 \]
 Let $\cpx{P}$ be the complex $0\ra P_1\raf{d} P_2\ra 0$, where $d$ is the unique
 (up to scalar) non-zero map, and $P_1$ is placed in degree zero.
 Then $\nu_{_A}(\cpx{P})$ is a complex of the form $0\ra I_1\ra I_2\ra 0$.
 The obvious surjective maps $P_1\ra I_1$ and $P_2\ra I_2$
 define a chain map from $\cpx{P}$ to $\nu_A(\cpx{P})$ and is a quasi-isomorphism.
 As a result, $\cpx{P}\in \nuinv{A}$, but $\cpx{P}$ does not belong to $\Kb{\stp{A}}$,
 since $\stp{A}=\{0\}$.

 \begin{Theo}\label{thm:restriction theorem}
  Let $A$ and $B$ be derived equivalent $k$-algebras, both of $\nu$-dominant
  dimension at least $1$.
  Then any derived equivalence $\mathcal{F}:\Db{A}\xrightarrow{\sim}\Db{B}$
  restricts to an equivalence
  \mbox{$\Kb{\stp{A}}\xrightarrow{\sim}\Kb{\stp{B}}$}
  of triangulated subcategories.
 \end{Theo}

 \begin{proof}
   Without loss of generality, we may assume that $\mathcal{F}$ is a standard
   derived equivalence.
   Then $\mathcal{F}$ induces
   $\Kb{\pmodcat{A}}\xrightarrow{\sim}\Kb{\pmodcat{B}}$ as
   triangulated subcategories, and
   $$\mathcal{F}(\nu_{_A}(\cpx{P}))\cong\nu_{_B}(\mathcal{F}(\cpx{P}))$$
   in $\Db{B}$, for any $\cpx{P}\in\Kb{\pmodcat{A}}$. Thus
   $\mathcal{F}(\nuinv{A})\subseteq \nuinv{B}$ and so
   \mbox{$\mathcal{F}(\Kb{\stp{A}})\subseteq \Kb{\stp{B}}$}
   by Proposition \ref{prop:nu-stable complexes}. Let $\mathcal{G}$ be a
   quasi-inverse of $\mathcal{F}$. The same arguments applied to  $\mathcal{G}$
   imply $\mathcal{G}(\Kb{\stp{B}})\subseteq \Kb{\stp{A}}$.
   Therefore, $\mathcal{F}$ induces an equivalence $\Kb{\stp{A}}\xrightarrow{\sim} \Kb{\stp{B}}$.
 \end{proof}

 \noindent \textit{Remark.} Theorem \ref{thm:restriction theorem} has two predecessors: A special case, stated only for gendo-symmetric algebras, was proved in \cite{FM}, where it was
 used to relate the Hochschild cohomology of $A$ and that of its associated self-injective algebra.
 A result in the same spirit as Theorem \ref{thm:restriction theorem} was obtained
 in \cite{HuXi10} without any restriction on algebras, but assuming $\mathcal{F}$
 to be an (iterated) almost $\nu$-stable derived equivalence.

 \begin{Koro}[(The derived restriction theorem)]
 \label{cor:restriction theorem for Morita algebras}
   Let $A$ and $B$ be derived equivalent $k$-algebras, both of $\nu$-dominant dimension
   at least $1$.
   Then the associated self-injective
   algebras of $A$ and $B$ are derived equivalent.
   In particular, every derived
   equivalence of Morita algebras induces a derived equivalence of their
   associated self-injective algebras.
 \end{Koro}

 \begin{proof}
   Let $H$ be an associated self-injective algebra of $A$. Then by definition
   $\stp{A}\cong \pmodcat{H}$ as additive categories, and thus
   $\Kb{\stp{A}}\cong \Kb{\pmodcat{H}}$ as triangulated categories. The
   statement then follows directly from Theorem \ref{thm:restriction theorem}.
 \end{proof}

An application of Corollary \ref{cor:restriction theorem for Morita algebras} is to
Auslander algebras:
 A (finite dimensional) $k$-algebra $A$ is said to be of
 \emph{finite representation type}, if there
 are only finitely many indecomposable $A$-modules (up to isomorphism).
 The {\em Auslander algebra} $\Gamma_A$ of $A$ is defined to be the endomorphism
 ring of the direct sum of all pairwise non-isomorphic indecomposable left
 $A$-modules.

 \begin{Koro} \label{cor:restriction theorem for Auslander algebras}
 Let $A$ and $B$ be self-injective $k$-algebras, both of finite representation
 type. Let $\Gamma_A$ and $\Gamma_B$ be the Auslander algebras of $A$ and
 $B$ respectively. Then

{\rm (1)} $\Gamma_A$ and $\Gamma_B$ are derived equivalent if and only if $A$
and $B$ are derived equivalent. 

{\rm (2)} If $\Gamma_A$ and $\Gamma_B$ are derived equivalent, then they are
   stably equivalent of Morita type.
 \end{Koro}

 \begin{proof}
   (1) If $A$ and $B$ are derived equivalent, then by \cite[Corollary 3.13]{HuXi13},
   their Auslander algebras $\Gamma_A$ and $\Gamma_B$ are derived equivalent. Conversely,
   if $\Gamma_A$ and $\Gamma_B$ are derived equivalent, then
   by Corollary \ref{cor:restriction theorem for Morita algebras} $A$ and $B$ are
   derived equivalent,
   since $A$ and $B$ are the associated self-injective algebras of $\Gamma_A$ and
   $\Gamma_B$ respectively.

   (2) If $\Gamma_A$ and $\Gamma_B$ are derived equivalent, then $A$
   and $B$ are derived equivalent by (1), and thus stably equivalent of Morita type by
   \cite[Corollary 5.5]{R91}. Now \cite[Theorem 1.1]{LXi05} implies
   that $\Gamma_A$ and $\Gamma_B$ are stably equivalent of Morita type.
 \end{proof}

\section{Invariance of homological dimensions}

In this section, the two invariance theorems will be proven. For almost
self-injective algebras the approach is to show that standard equivalences
have a special form; they are iterated almost $\nu$-stable and therefore
preserve both global and dominant dimension. For algebras with a duality, a
rather different approach will be taken, identifying the two homological
dimensions inside the derived category. Before addressing invariance, in
the first subsection the general question is addressed how much homological
dimensions can vary under derived equivalences.

\subsection{Variance of homological dimensions under derived
  equivalences}\label{subsec:bound the difference}

 As it is well-known, the difference of global dimensions of two derived
 equivalent algebras is bounded by the
 length of a tilting complex inducing a derived equivalence,
 see for example \cite[Section 12.5(b)]{GR}.
 More precisely,
 let $A$ be a $k$-algebra, and define the \emph{length}
 of a radical complex $\cpx{X}$ in $\Kb{A}$ to be
 \[
    \ell(\cpx{X}) = \sup\{t\mid X^t\neq 0\} - \inf\{b\mid X^b\neq 0\} +1.
 \]
 The length of an arbitrary complex $\cpx{Y}$ in $\Kb{A}$ is defined to
 be the length of the unique radical complex that is isomorphic to $\cpx{Y}$ in
 $\Kb{A}$ (Lemma \ref{lem:morphisms in derived categories}).

 \begin{Prop}[{(\cite[Section 12.5(b)]{GR})}]
   \label{prop:bound global dimension}
   Let $\mathcal{F}: \Db{A}\lraf{\sim} \Db{B}$ be a derived equivalence
   between $k$-algebras. Then $|\gldim(A)-\gldim(B)|\leq
   \ell(\mathcal{F}(A))-1$.
 \end{Prop}

 Naively, one may expect that dominant dimension behaves similarly
 under derived equivalences. Here is a counterexample:
 Let $n\geq 2$ be an integer, and let $A$ be the $k$-algebra given by the
 quiver
 \[
   \xymatrix{
   1&  2\ar[l]_{\alpha_1}& 3\ar[l]_{\alpha_2}& \cdots \ar[l]_{\alpha_3}& 2n\ar[l]_(.45){\alpha_{2n-1}}& 2n+1\ar[l]_{\alpha_{2n}}
   }
 \]
 and relations $\alpha_{i}\alpha_{i+1}=0\; (1\leq i\leq 2n-1, i\neq n)$.
 Let $S_i$ denote the simple left $A$-module corresponding to
 the vertex $i$ and $P_i$ be the projective cover of $S_i$.
 The modules $P_i$ are projective-injective
 for $ i\neq 1,n+1$. The projective dimensions of the simple modules are
 $\projdim S_1 = 0$, $\projdim S_{i} = \projdim S_{i+n} = i-1$ for
 $2\leq i\leq n+1$, and the minimal injective resolution of
 the left regular $A$-module is of the form
 \begin{align*}
    0\ra A \ra (\bigoplus_{i\neq 1, n+1} P_i)\oplus P_2\oplus P_{n+2}
    \ra P_3\oplus P_{n+3}\ra \cdots & \ra P_n \oplus P_{2n}
    \ra P_{n+2}\oplus P_{2n+1}\\
    & \ra S_{2n+1}\oplus P_{n+2}/S_n\ra 0
 \end{align*}
 Consequently, $\domdim(A)= \gldim(A)=n$.
 Let $T:=\tau^{-1}S_1\oplus P_2\oplus\cdots P_{2n+1}$ be the APR-tilting
 module (see, for instance, \cite[VI.2.8]{ASS}) 
 associated with the projective simple $A$-module $S_1$, and
 let $B=\End_A(T)$. Then $\projdim T=1$ and therefore the derived
 equivalence between $A$ and $B$ induced by $T$ is given by a two-term
 tilting complex. By direct computation, $B$ is seen to be
 isomorphic to the $k$-algebra given by the same quiver as $A$
 but with different relations $\alpha_{i}\alpha_{i+1}=0$ for $2\leq i\leq 2n,
 \, i\neq n$.
 As a result, $\domdim(B) = 1$ and the difference between
 dominant dimensions of $A$ and $B$ is $(n-1)$, although the derived
 equivalence is induced by a tilting module.

 Although this example smashes any hope to bound the difference
 of dominant dimensions of derived equivalent algebras in terms of lengths
 of tilting complexes, there are still some cases where both
 global dimension and dominant dimension behave nicely, see
 \cite{HuXi10,Hu12} and \cite{FM}. Note that both algebras in the example
 above are of $\nu$-dominant dimension $0$. This suggests to restrict
 attention to algebras of $\nu$-dominant dimension
 at least $1$ - a restriction that has been needed for
 Theorem \ref{thm:restriction theorem} and
that will be further justified by the invariance results later on.

 \begin{Theo}\label{thm:bound the difference}
   Let $A$ and $B$ be $k$-algebras, both of $\nu$-dominant
   dimension at least $1$ and derived equivalent by
   $\mathcal{F}: \Db{A}\to \Db{B}$. Then
   $|\domdim(A)-\domdim(B)|\leq \ell(\mathcal{F}(A))-1$.
 \end{Theo}

 \begin{proof}
   Let $m=\domdim(B)$ and $n=\ell(\mathcal{F}(A))-1$.
   Since $\ell(\mathcal{F}^{-1}(B))=\ell(\mathcal{F}(A))=n+1$ by \cite[Lemma 2.1]{HuXi10},
   it is enough to show that $\domdim(A)\geq m-n.$
   If $m\leq n+1$, there is nothing left to prove. Assume that $\mathcal{F}$
   is a standard derived equivalence,
   $m>n+1$ and $\cpx{P} = \mathcal{F}(A)$ is a radical tilting complex
   in $\Kb{\pmodcat{B}}$ of the following form (up to degree shift)
   \[
    0\lra P^0\lra P^1\lra P^2\lra \cdots \lra P^{n-2}\lra P^{n-1}\lra P^n \lra 0
   \]
   where $P^0$ is nonzero and placed in degree $0$. Take a Cartan-Eilenberg
   injective resolution $\cpx{I}$ of $\cpx{P}$ that is the total complex
   of the double complex obtained by taking minimal injective resolutions
   of $P^i$ for all $i$.
   Since $\Ndomdim(B)\geq 1$ and $m\geq n+1$, the modules $I^i$ are strongly
   projective-injective for $0\leq i\leq m-1$
   by Lemma \ref{lem:two dominant dimensions coincide}.
   Let $\cpx{I}_{<m}$ and $\cpx{I}_{\geq m}$ be the brutal truncations of
   $\cpx{I}$.
   By definition of brutal truncation, there is a triangle in $\D{B}$
   \[
    (\ast) \qquad \cpx{I}_{<m}[-1]\lra \cpx{I}_{\geq m}\lra \cpx{I}\lra \cpx{I}_{<m}.
   \]
   The standard derived equivalence $\mathcal{F}$ lifts
   to a derived equivalence (denoted by $\mathcal{F}$ again)
   from $\D{A}$ to $\D{B}$.
   Applying $\mathcal{F}^{-1}$ to the triangle $(\ast)$, we
   obtain the  triangle in $\D{A}$
   \begin{align}\label{eqn:key  triangle II}
    \mathcal{F}^{-1}(\cpx{I}_{<m})[-1]\lra  \mathcal{F}^{-1}(\cpx{I}_{\geq m})\lra \mathcal{F}^{-1}(\cpx{I})\lra \mathcal{F}^{-1}(\cpx{I}_{<m})
   \end{align}
   and therefore the long exact sequence
   \begin{align}\label{eqn:long exact sequence}
     \cdots\ra \mathrm{H}^{i-1}(\mathcal{F}^{-1}(\cpx{I}_{<m}))\ra
     \mathrm{H}^i(\mathcal{F}^{-1}(\cpx{I}_{\geq m}))\ra
     \mathrm{H}^i(\mathcal{F}^{-1}(\cpx{I}))\ra \mathrm{H}^i(\mathcal{F}^{-1}(\cpx{I}_{<m})) \ra \cdots
   \end{align}
   Note that $\mathcal{F}^{-1}(\cpx{I})
   \cong \mathcal{F}^{-1}(\cpx{P})\cong \mathcal{F}^{-1}\circ \mathcal{F}(A)\cong A$ in $\D{A}$,
   and $\mathcal{F}^{-1}(\cpx{I}_{<m})$ belongs to $\Kb{\stp{A}}$ by
   the construction of $\cpx{I}_{<m}$ and Theorem \ref{thm:restriction theorem}.
   Moreover, Lemma \ref{lem:morphisms in derived categories}
   implies for $i\leq m-n-1$,
   \begin{align*}
    \mathrm{H}^{i}(\mathcal{F}^{-1}(\cpx{I}_{\geq m})) \cong
    \Hom_{\D{A}}(A, \mathcal{F}^{-1}(\cpx{I}_{\geq m})[i]) \cong
    \Hom_{\D{B}}(\cpx{P}, \cpx{I}_{\geq m}[i]) =0
   \end{align*}
   since $(\cpx{I}_{\geq m}[i])^p =0$ for $p\leq n$ and $i\leq m-n-1$.
   Therefore, from the long exact sequence (\ref{eqn:long exact sequence}),
   $\mathrm{H}^i(\mathcal{F}^{-1}(\cpx{I}_{<m}))=0$ for $i<0$,
   $\mathrm{H}^0(\mathcal{F}^{-1}(\cpx{I}_{<m}))\cong A$, and
   \mbox{$\mathrm{H}^{i}(\mathcal{F}^{-1}(\cpx{I}_{<m}))\cong 
\mathrm{H}^{i+1}(\mathcal{F}^{-1}(\cpx{I}_{\geq m})) = 0$}
   for $1\leq i\leq m-n-2$.
   Hence
   $\mathcal{F}^{-1}(\cpx{I}_{<m})$ is isomorphic to a radical complex
   in $\Kb{\stp{A}}$ of the form
   \[
    0\lra E^0 \lra E^1 \lra \cdots \lra E^{m-n-2}\lra E^{m-n-1}\lra \cdots
   \]
   such that $0\to A\to E^0\ra E^1\ra \cdots \ra E^{m-n-2}\ra E^{m-n-1}$
   is exact. Consequently the dominant dimension of $A$ is at least $m-n$,
   as desired.
 \end{proof}

A special case of Theorem \ref{thm:bound the difference} is:

 \begin{Koro}\label{cor:bound difference for morita algebras}
    Let $A$ and $B$ be Morita algebras. If there is a derived equivalence 
\mbox{$\mathcal{F}: \Db{A}\raf{\sim} \Db{B}$},
    then $|\domdim(A)-\domdim(B)|\leq \ell(\mathcal{F}(A))-1$.
 \end{Koro}


\subsection{Almost self-injective algebras}

  Interactions between derived equivalences
  and stable equivalences frequently seem to be of particular interest,
  see \cite{HuXi10,HuXi}
  and the references therein. \cite[Corollary 5.5]{R91} and Corollary
  \ref{cor:restriction theorem for Auslander algebras} state that,
  for self-injective
  algebras and Auslander algebras of finite representation type self-injective
  algebras, derived equivalences imply stable equivalences of Morita
  type. The following theorem implies that the same holds for almost
  self-injective algebras, by characterising all
  derived equivalences among them as
  iterated almost $\nu$-stable derived equivalences.

 \begin{Theo}\label{thm:derived equivalences for almost self-injective}
   Let $A$ and $B$ be derived equivalent almost self-injective algebras.
   Then any standard derived equivalence
   $\mathcal{F}:\Db{A}\xrightarrow{\sim}\Db{B}$
   is an iterated almost $\nu$-stable
  derived equivalence (up to shifts). In particular, derived equivalent almost
  self-injective algebras are stably equivalent of Morita type.
 \end{Theo}

 To prove this theorem, we need some preparations. First, we recall
 some basics on $\mathcal{D}$-split sequences introduced in \cite{HuXi11},
 see also \cite{HuXi13}.
 Let $\cal A$ be an additive category and let $\mathcal{X}$ be a full
 subcategory of $\cal A$. A morphism
 $f:X\to M$ in $\cal A$ is a \emph{right $\mathcal{X}$-approximation}, if
 $X\in \mathcal{X}$ and for any $X'\in \mathcal{X}$, the canonical
 morphism $\Hom_{\cal A}(X',X)\to \Hom_{\cal A}(X',M)$ is an epimorphism.
 We call $f$ \emph{right minimal} if an equality $\alpha \cdot f =f$
 implies that $\alpha$ is an isomorphism, for 
\mbox{$\alpha\in \End_{\cal A}(X)$.}
 Left $\mathcal{X}$-approximations and left minimal morphisms are defined
 similarly.
 Let ${\cal C}$ be a triangulated category
 and let ${\cal D}$ be a full (not necessarily triangulated) additive subcategory of ${\cal C}$.
 A triangle in ${\cal C}$
 $$X\lraf{f}D\lraf{g}Y\lraf{h} X[1]$$ is called
 a \emph{${\cal D}$-split triangle} if $f$ is a left
 ${\cal D}$-approximation and $g$ is a right ${\cal D}$-approximation.
 A full subcategory ${\cal T}$ of ${\cal C}$ is called a {\em tilting
   subcategory}
 if $\Hom_{\cal C}({\cal T}, {\cal T}[i])=0$ for all $i\neq 0$ and $\cal C$
 itself is the only triangulated subcategory of $\cal C$ that
 contains $\cal T$ and is closed under taking direct summands.
 An object $T$ in ${\cal C}$ is a {\em tilting object}
 if $\add(T)$ is a tilting subcategory of ${\cal C}$.
 For example, all tilting complexes over an algebra $A$ are tilting objects
 in $\Kb{\pmodcat{A}}$.

 \begin{Lem} \label{lem:D-split sequences}
   Let ${\cal C}$ be a triangulated category, and ${\cal D}$ an additive full
  subcategory of ${\cal C}$.
  Let $X\raf{f}D\raf{g}Y\raf{h} X[1]$ be a ${\cal D}$-split triangle. Then: 
  
    {\rm (1)} Suppose that $f$  is left minimal and $g$ is right minimal,
    and that $X\cong \bigoplus_{i=1}^nX_i$ and $Y\cong \bigoplus_{i=1}^mY_i$
    are decompositions of $X$ and $Y$ into strongly indecomposable direct
    summands.
    Then $m=n$. In particular, if indecomposable objects in ${\cal C}$ are
    strongly indecomposable, then $X$ is indecomposable if and only if so
    is $Y$.

    {\rm (2)} If ${\cal D}\cup \{ X\}$ is a tilting subcategory of $\cal C$ and
    $\Hom_{\cal C}({\cal D}, f)$ is injective, then ${\cal D}\cup \{Y\}$ is also
    a tilting subcategory of $\cal C$. 
    
    {\rm (3)} If ${\cal D}\cup \{Y\}$ is a tilting subcategory of $\cal C$ and
    $\Hom_{\cal C}(g, {\cal D})$ is injective, then ${\cal D}\cup \{X\}$ is also
    a tilting subcategory of $\cal C$.
 \end{Lem}

 \begin{proof}
  (2) and (3) follow from \cite[Theorem 2.32]{AI}. It remains to prove (1).
  If $D\cong 0$, then \mbox{$Y\cong X[1]$} 
and we are done. Now we assume that $D\ncong 0$.
  Then $X\ncong 0$ and $Y\ncong 0$, since otherwise $f=0$ or $g=0$, which
  contradicts the minimality of $f$ and $g$. To proceed, let
  $I_X=\{\alpha\in \End_{\cal C}(X)\mid \text{$\alpha$ factors through $f$}\}$ and
  $I_Y=\{\gamma\in \End_{\cal C}(Y)\mid \text{$\gamma$ factors through $g$}\}$.

\smallskip
 {\it Claim.
 $I_X$ is a two-sided ideal of $\End_{\cal C}(X)$ and it is contained in the
 Jacobson radical $J_X$ of $\End_{\cal C}(X)$.}

\smallskip
  {\it Proof.} For any $\alpha\in I_X$ and $\theta\in \End_{\cal C}(X)$,
  there exist $u\in \Hom_{\cal C}(D,X)$ with $\alpha = f\cdot u$, and
  $\omega\in \End_{\cal C}(D)$ with $\theta\cdot f = f\cdot \omega$
  since $f$ is a left $\mathcal{D}$-approximation.
  Therefore,
  $\theta\cdot \alpha = \theta \cdot f\cdot u = f\cdot (\omega\cdot u)$, which
  implies $\theta\cdot \alpha\in I_X$. Similarly $\alpha\cdot \theta\in I_X$ and
  so $I_X$ is a two-sided ideal of $\End_{\cal C}(X)$. To see $I_X\subseteq J_X$,
  it suffices to show, by \cite[Theorem 15.3, p. 166]{AF}, that $1-\alpha$ is
  invertible in $\End_{\cal C}(X)$ for each $\alpha\in I_X$. Let $\alpha = f\cdot u$
  for some $u\in \Hom_{\cal C}(D,X)$.
  In the diagram in $\cal C$
  \[
    \xymatrix{Y[-1]\ar[r]^{-h[-1]} \ar@{=}[d]^{\mathrm{id}} & X \ar[r]^f \ar[d]^{1-\alpha} & D \ar[r]^g \ar@{-->}[d]^\beta  & Y \ar@{=}[d]^{\mathrm{id}}\\
              Y[-1]\ar[r]^{-h[-1]}                      & X \ar[r]^f               & D \ar[r]^g               & Y}
  \]
  the first square commutes because $h[-1]\cdot \alpha = h[-1]\cdot f\cdot u=0$,
  and $\beta$ exists so that the other squares commute 
by axioms of triangulated categories.
  Since $g$ is a right minimal morphism, $\beta$ is an isomorphism,
  and thus again by axioms of triangulated categories, $1-\alpha$ is an isomorphism.

  By similar arguments, $I_Y$ is a two-sided ideal of $\End_{\cal C}(Y)$
  and it is contained in the Jacobson radical $J_Y$ of $\End_{\cal C}(Y)$.

\smallskip
  {\it Claim. There is an algebra isomorphism $\End_{\cal C}(X)/I_X\cong \End_{\cal C}(Y)/I_Y$. }
  
\smallskip
 {\it Proof.}
  We first construct a ring homomorphism $\phi: \End_{\cal C}(X)\to
  \End_{\cal C}(Y)/I_Y$ as follows. For each $\alpha\in \End_{\cal C}(X)$,
  there exists a commutative diagram in $\cal C$
  \[
    \xymatrix{X \ar[r]^f \ar[d]^\alpha  & D \ar[r]^g \ar@{.>}[d]^\beta & Y \ar[r]^h \ar@{-->}[d]^\gamma & X[1]\ar[d]^{\alpha[1]}\\
              X \ar[r]^f                & D \ar[r]^g                   & Y \ar[r]^h                     & X[1]}
  \]
  where $\beta$ exists since $f$ is a left $\mathcal{D}$-approximation,
  and $\gamma$ exists by axioms of triangulated categories. If
  $\beta'\in \End_{\cal C}(D)$ and $\gamma'\in \End_{\cal C}(Y)$
  are different choices such that $f\cdot \beta' = \alpha\cdot f$
  and $g\cdot \gamma' = \beta'\cdot g$, then
  $(\gamma-\gamma')\cdot h=0$ which implies that $\gamma-\gamma' = u\cdot g$
  for some $u\in \Hom_{\cal C}(Y,D)$. In other words, the image of
  $\gamma$ in the quotient ring $\End_{\cal C}(Y)/I_Y$ is well-defined,
  and thus $\phi(\alpha)$ is well-defined.
  Then $\phi$ is a ring homomorphism. It is surjective since $g$ is a right
  $\mathcal{D}$-approximation.

\smallskip
  {\it Claim.} The kernel of $\phi$ equals $I_X$. 
  
\smallskip
{\it Proof.}
  For any $\alpha\in I_X$, there exists $v\in \Hom_{\cal C}(D,X)$ with
  $\alpha  = f\cdot v$.
  So $\phi(\alpha)\cdot h =$ \mbox{$h\cdot \alpha[1]$} 
$= h\cdot f[1]\cdot v[1]=0$
  which implies that $\phi(\alpha)$ factors through $g$ and thus $I_X\subseteq \ker(\phi)$.
  On the other hand, for any \mbox{$\alpha\in \ker(\phi)$}, there exists
  $u\in \Hom_{\cal C}(Y,D)$ with $\phi(\alpha) = u\cdot g$, and so
  $h[-1]\cdot \alpha = \phi(\alpha)[-1]\cdot h[-1] = (u\cdot g\cdot h)[-1]=0$
  which implies that $\alpha$ factors through $f$ and thus 
\mbox{$\ker(\phi)\subseteq I_X$}.

\smallskip
  Altogether, $\phi$ induces an isomorphism
  $\End_{\cal C}(X)/I_X\cong \End_{\cal C}(Y)/I_Y$, and hence further an
  isomorphism $$\End_{\cal C}(X)/J_X\cong \End_{\cal C}(Y)/J_Y.$$
  Since both $X$ and $Y$ are decomposed into strongly indecomposable
  direct summands, it follows that
  both rings $\End_{\cal C}(X)$ and $\End_{\cal C}(Y)$ are semi-perfect rings
  by \mbox{\cite[Theorem 27.6(b), p. 304]{AF}}, and therefore the isomorphism
  above implies that $\End_{\cal C}(X)$ and $\End_{\cal C}(Y)$ have
  the same number of simple left modules, or equivalently the same
  number of indecomposable projective left modules. Consequently, $X$
  and $Y$ have the same number of indecomposable direct summands, that
  is $m=n$.
 \end{proof}

 \begin{Lem} \label{lem:tilting complex over almost self-injective algebras}
  Let $A$ be an almost self-injective algebra, but not self-injective.
  Let $E$ be an additive generator of $\stp{A}$ and let $P$ be
  the unique indecomposable projective left $A$-module
  such that $\pmodcat{A}=\add{(P\oplus E)}$.
  Let $\cpx{T}=(T^i,d^i)$ be an indecomposable radical complex in $\Kb{\pmodcat{A}}$.
  If $\cpx{T}\oplus E$ is a tilting complex over $A$, then at least one of the
  following two assertions holds true:

    {\rm (1)} There exists $r\leq 0$ such that $T^r\cong P$, $T^i=0$ for all $i>0$
  and for all $i < r$, and $T^i \neq 0$ in $\stp{A}$
    for all $r< i \leq 0$;

    {\rm (2)} There exists $s\geq 0$ such that $T^s\cong P$, $T^i=0$ for all
    $i<0$ and for all $ i > s$, and $T^i \neq 0$ in $\stp{A}$
    for all $0\leq i<s$.
 \end{Lem}


 \begin{proof}
   {\it Special case.} The complex $\cpx{T}$ has only one nonzero term,
   that is, $\cpx{T} = Q[m]$
   for some $Q\in \pmodcat{A}$ and $m\in \mathbb{Z}$. Then
   $Q$ is indecomposable since $\cpx{T}$ is indecomposable.
   If \mbox{$Q\in \stp{A}$}, then
   $\cpx{T}\oplus E\in \Kb{\stp{A}}\subsetneq \Kb{\pmodcat{A}}$
   which contradicts the assumption that $\cpx{T}\oplus E$, as a tilting complex, generates
   $\Kb{\pmodcat{A}}$.
   Since each indecomposable projective left
   $A$-module is either isomorphic to $P$ or strongly projective-injective,
   by the definition of almost self-injective algebras
   $Q$ must be isomorphic to $P$. We still have to show that $m=0$.
   Assume $m\neq 0$. Then
   by the self-orthogonality of tilting complexes,
   \mbox{$\Hom_{\Kb{\pmodcat{A}}}(P[m]\oplus E, E[m])=0$},
   which implies $\Hom_A(P,E)=0$, a contradiction to $\Ndomdim(A)\geq 1$.
   So, $\cpx{T}=P$, which satisfies both conditions (1) and (2).

   In the {\it general case},
   $\cpx{T}=(T^i, d^i)$ is an indecomposable radical complex
   of the following form
   \begin{align*}
   0\lra T^r\lraf{d^r} T^{r+1}\lra \cdots \lraf{d^{s-1}} T^s\lra 0
   \end{align*}
   where $T^i$ are nonzero projective left $A$-modules and $r<s$.
   Using the self-orthogonality of the
   tilting complex $\cpx{T}\oplus E$, we are going to check:

\smallskip
   {\it Claim. {\rm (a)} $T^r$ and $T^s$ have no nonzero common direct summands;
   
   {\rm (b)} if $r\neq 0$, then $T^r\in \add(P)$; and
   
  {\rm (c)} if $s\neq 0$, then $T^s\in \add(P)$.}

\smallskip
   {\it Proof.} To see (a), assume on the contrary
   that $K$ is a nonzero common direct summand
   of both $T^r$ and $T^s$. Let $u$ be a split epimorphism
   from $T^r$ to $K$ and $v$ be a split monomorphism from $K$
   to $T^s$. Then the composition $u\cdot v$ from $T^r$ to $T^s$
   defines a nonzero morphism in $\Hom_{\Kb{\pmodcat{A}}}(\cpx{T}, \cpx{T}[s-r])$,
   because $\cpx{T}$ is a radical complex. But this contradicts
   the assumption that $\cpx{T}$ is self-orthogonal and $s-r>0$, which forces
   any morphism from $\cpx{T}$ to $\cpx{T}[s-r]$ in $\Kb{\pmodcat{A}}$
   to be zero. 
   
   Similarly, (b) and (c) follow since for any $m\neq 0$,
   \[
   \Hom_{\Kb{\pmodcat{A}}}(\cpx{T},E[m])=0=
   \Hom_{\Kb{\pmodcat{A}}}(E, \cpx{T}[m]).
   \]

   As a consequence of the claim, $s=0$ or $r=0$. Indeed, if
   $r\neq 0$ and $s\neq 0$, then by (b) and (c),
   $T^r$ and $T^s$ have a common direct summand $P$, which contradicts
   (a). 

   We will finish the proof by analysing these two cases.

   \indent \textit{Case $\mathit{r<s=0}$.} By (b), $T^r\in \add(P)$ and
   then $T^0\in \add(E)$  by (a). Let $\cpx{T}_{<0}$ be
   the brutal truncation of $\cpx{T}$, which by definition provides
   the following
   triangle in $\Kb{\pmodcat{A}}$:
   \begin{align} \label{eqn:key  triangle}
    \cpx{T}_{<0}[-1]\lraf{f} T^0 \lraf{g} \cpx{T}\lra \cpx{T}_{<0}
   \end{align}
  where $f$ is the chain morphism induced by $d^{-1}:T^{-1}\to T^0$,
  and $g$ is the chain morphism induced by $\mathrm{id}: T^0\to T^0$.
  Applying $\Hom_{\K{A}}(-,E)$ to this  triangle gives
  the short exact sequence
  \begin{align*}
    0 \ra \Hom_{\K{A}}(\cpx{T},E)\raf{g^*} \Hom_{\K{A}}(T^0,E)
    \raf{f^*} \Hom_{\K{A}}(\cpx{T}_{<0}[-1],E)\ra 0
  \end{align*}
  since $\Hom_{\K{A}}(\cpx{T}_{<0}, E)=0$ trivially and
  $\Hom_{\K{A}}(\cpx{T}[-1],E)=0$ by $\cpx{T}\oplus E$ being tilting.
  In particular, $g^*$ is injective, and $f$ is a left $\add(E)$-approximation.

\smallskip
  {\it Claim.} The morphism $f$ is left minimal. 
  
\smallskip
   {\it Proof.}
  For any $u:T^0\to T^0$ with $f\cdot u = f$, we have
  $d^{-1}\cdot u = d^{-1}$. Thus $u$ defines a chain
  morphism
  $\cpx{\gamma}: \cpx{T}\to \cpx{T}$ in $\Kb{\pmodcat{A}}$
  with $\gamma^0=u$ and
  $\gamma^i = \mathrm{id}$ for all $i\neq 0$. Now $\cpx{T}$ being
  an indecomposable complex in $\Kb{\pmodcat{A}}$ implies that
  $\End_{\Kb{\pmodcat{A}}}(\cpx{T})$ is a local $k$-algebra, and therefore
  $\cpx{\gamma}$ is either nilpotent or invertible.
  If $\cpx{\gamma}$ is nilpotent in $\End_{\Kb{\pmodcat{A}}}(\cpx{T})$, say
  ${(\cpx{\gamma})}^m=0$ for some $m\in \mathbb{Z}$,
 then there exists a homotopy
 morphism \mbox{$\cpx{s} = \{s^i\}:\cpx{T}\ra \cpx{T}[-1]$}
 such that $(\gamma^i)^m = s^i\cdot d^{i-1}+d^i\cdot s^{i+1}$. By
 definition of $\cpx{T}$, the map $d^{r-1}$ is zero.
  Thus, \mbox{$\mathrm{id} = d^r\cdot s^{r+1}$} and hence
  $d^{r}$ is a split monomorphism. But this contradicts
  $\cpx{T}$ being a radical complex. It follows that $f$
  is a left minimal morphism.

\smallskip
  Applying $\Hom_{\K{A}}(E,-)$ to the  triangle (\ref{eqn:key triangle})
  implies that $g$ is a right $\add(E)$-approximation. Since $g$ 
  induces an isomorphism
  $\Hom_{\K{A}}(T^0, \cpx{T})\cong \Hom_A(T^0,T^0)$, $g$
  is a right minimal morphism.
  Consequently, the  triangle (\ref{eqn:key  triangle}) is
  a $\add(E)$-split sequence, with $g^*$ being an injective morphism.
  Thus, by Lemma \ref{lem:D-split sequences} (1) and (3),
  $\cpx{T}_{<0}$ is an indecomposable radical complex, and
  $\cpx{T}_{<0}[-1]\oplus E$
  is a tilting complex over $A$. Replacing $\cpx{T}$ by $\cpx{T}_{<0}[-1]$
  and repeating the same arguments $(r-1)$ times, we finally get that
  $T^r$ is indecomposable and $T^r\oplus E$ is a tilting complex over $A$,
  whereby $T^r$ is isomorphic to $P$ and the assertion (1)
  holds.

  \indent \textit{Case $0 = r<s$.} By similar arguments as in the
  previous case, the assertion (2)
  in the statement can be verified for $\cpx{T}$.

  Since the case $r=s=0$ has been shown already, all cases have been settled.
 \end{proof}

 With these preparations, we are now able to
 prove Theorem \ref{thm:derived equivalences for almost self-injective}.

 \begin{proof}[of Theorem
\ref{thm:derived equivalences for almost self-injective}]


First we assume that one of $A$ or $B$ is self-injective.
When the ground field is algebraically closed, being self-injective is
invariant under derived equivalence, by \cite{AR04}. Under our assumption
$\Ndomdim \geq 1$, invariance of being self-injective can be shown as follows,
without any restriction on the ground field. By assumption,
both $A$ and $B$ have $\Ndomdim \geq 1$. Then $A$ is self-injective if
and only if $\Kb{\pmodcat{A}} = \Kb{\stp{A}}$, and similarly for $B$.
By Theorem \ref{thm:restriction theorem}, the given derived
equivalence $\mathcal{F}$ induces an equivalence $\Kb{\stp{A}}\raf{\sim}
\Kb{\stp{B}}$ of triangulated subcategories.
Hence, both $A$ and $B$ are self-injective.
   Now, \cite[Proposition 3.8]{HuXi11} implies that $\mathcal{F}$
   is an almost $\nu$-stable derived equivalence up to degree shift.

Next, we assume that neither $A$ nor $B$ is self-injective.
   Let $E_A$ and $E_B$ be basic additive generators of $\stp{A}$
   and $\stp{B}$ respectively. Then the number of indecomposable
   direct summands of $E_B$ is exactly one less than the number
   of the simple left $B$-modules.
   Since $\Ndomdim(A)\geq 1$ and $\Ndomdim(B)\geq 1$,
   the derived equivalence $\mathcal{F}$ induces
   an equivalence $\Kb{\stp{A}}\raf{\sim} \Kb{\stp{B}}$ of triangulated subcategories,
   again by Theorem \ref{thm:restriction theorem}.
   Let $\cpx{\mathcal{E}}$ be the image of $E_A$ in  $\Kb{\pmodcat{B}}$
   under the equivalence.
   Then $\add(\cpx{\mathcal{E}})$ generates $\Kb{\stp{B}}$, and
   $\Hom_{\Kb{\stp{B}}}(\cpx{\mathcal{E}}, \cpx{\mathcal{E}}[i])=0$ unless $i=0$.
   Without loss of generality, we assume that $\cpx{\mathcal{E}}$ is,
   up to degree shift, a radical complex of the form
   $$\cpx{\mathcal{E}}:\qquad 0\lra S^{-n}\lra\cdots\lra S^{-1}\lra S^0\lra 0.$$
   where $S^i\in \add(E_B)$ for all $i$.
   By \cite[Proposition 4.1]{HuXi},
   there exists a complex $\cpx{X}$ in
   $\Kb{\pmodcat{B}}$ such that $\cpx{X}\oplus \cpx{\mathcal{E}}$ is a
   basic tilting complex over $B$, inducing an almost $\nu$-stable
   derived equivalence $$\mathcal{G}: \Db{B}\lraf{\sim} \Db{C}$$ where
   $C=\End_{\Kb{\pmodcat{B}}}(\cpx{X}\oplus \cpx{\mathcal{E}})$.
   Let $\cpx{T}$ be a radical complex such that $\mathcal{F}(\cpx{T})\cong \cpx{X}$.
   Then $(\mathcal{G}\circ \mathcal{F})(\cpx{T}\oplus E_A)\cong C$.
   In particular, $\cpx{T}\oplus E_A$ is a tilting complex over $A$
   and $\cpx{T}$ is indecomposable.
   Hence by Lemma \ref{lem:tilting complex over almost self-injective algebras} and
   \cite[Theorem 1.3(5)]{Hu12}, $\mathcal{G}\circ \mathcal{F}$ is an
iterated almost $\nu$-stable
   derived equivalence up to degree shift,
and thus $\mathcal{F} = \mathcal{G}^{-1}\circ (\mathcal{G}
   \circ \mathcal{F})$ is an iterated almost $\nu$-stable derived equivalence.
 \end{proof}


 \begin{Koro}
 Let $A$ and $B$ be self-injective $k$-algebras and let $X$ (respectively $Y$)
 be an indecomposable left $A$-module (respectively left $B$-module).
 If $\End_A(A\oplus X)$ and $\End_B(B\oplus Y)$ are derived equivalent, then
 every standard derived equivalence between them
 is an iterated almost $\nu$-stable derived
 equivalence (up to degree shift), and the two endomorphism
 algebras are stably equivalent of Morita type.
 \end{Koro}

 \begin{proof}
   By Definition \ref{def:morita algebras}, both endomorphism
   algebras are almost self-injective algebras. The statements then follow directly
   from Theorem \ref{thm:derived equivalences for almost self-injective}.
 \end{proof}

 \noindent \textit{Example.}
 Theorem \ref{thm:derived equivalences for almost self-injective}
 may fail if the algebras $A$ and $B$ are not assumed to be almost
 self-injective. Here is an example.
 Let $\Lambda=k[x,y]/(x^2,y^2)$, and let $S$
 be the unique simple $\Lambda$-module.
 Then $\Lambda$ is a self-injective $k$-algebra and the
 Auslander-Reiten sequence
 \[
   0\ra \Omega^2 S\ra \rad(\Lambda)\oplus\rad(\Lambda)\ra S\ra 0
 \]
 is an $\add(\Lambda\oplus\rad(\Lambda))$-split sequence in the
 sense of \cite{HuXi11}.
 By \cite[Theorem 1.1]{HuXi11}, the endomorphism algebras
 $A:=\End_{\Lambda}(\Lambda\oplus\rad(\Lambda)\oplus S)$ and
 $B:=\End_{\Lambda}(\Lambda\oplus\rad(\Lambda)\oplus \Omega^2S)$ are
 derived equivalent.
 By direct checking, both $A$ and $B$ are seen to
 have $\nu$-dominant dimension at least two, but
 none of them is an almost self-injective algebra. Since
 $\gldim (A)=2$ and
 $\gldim(B) =3$, the algebras $A$ and $B$ cannot be
 stably equivalent of Morita type.

\medskip
 Combining Proposition \ref{prop:almost nu-stable preserve dimensions} and
 Theorem \ref{thm:derived equivalences for almost self-injective} yields
 the {\em First Invariance Theorem}:

 \begin{Koro}\label{cor:class almost self-injective algebras}
   Derived equivalences between almost self-injective algebras preserve both
   global dimension and dominant dimension.
 \end{Koro}

 \subsection{Algebras with anti-automorphisms preserving simples}

 Many algebras have anti-automorphisms that preserve simple modules.
 Prominent examples are involutory such anti-automorphisms, often called
 dualities, which are fundamental ingredients of the definition of
 cellular algebras. Among the examples are group algebras of symmetric
 groups in any characteristic, many Hecke algebras, Brauer algebras,
 Temperley-Lieb algebras and many classes of Schur algebras. In particular,
 quasi-hereditary algebras with duality occur frequently in algebraic Lie
 theory. They have been studied in \cite{FK11a,FK15,FM} from the point of view
 of dominant dimension, motivating the concept of gendo-symmetric algebras
 and our present investigation of invariance properties. The
 anti-automorphisms defining dualities often occur as shadows of Cartan
 involutions. A much larger class of algebras (with anti-automorphisms not
 required to be involutions)
 will be shown now to satisfy homological invariance
 properties. After establishing various elementary properties of the class of
 algebras with anti-automorphism preserving simples, 
 both global and dominant dimension will get
 identified explicitly in the derived category, in terms of complexes
 occurring in canonical triangles.

 Now we state precisely what we mean by an anti-automorphism preserving
simples, and then we collect
 basic properties of algebras with such anti-automorphisms.
See also \cite{CPS96,FK11a,MO04}
 and the references therein for further discussion of the special case of
dualities.
 Let $A$ be a $k$-algebra and let
$\gamma: A\to A$ be an algebra anti-automorphism.
 For each (left) $A$-module $M$, the $\gamma$-twist of $M$, denoted by
 $^\gamma\!M$, is defined to be
 the right $A$-module that equals $M$ as a $k$-vector space,
 and is equipped with the right $A$-action:
 \[
    m\cdot a = \gamma(a)m\qquad \forall\ a\in A, m\in M
    \]
 The notation $\cdot$ is reserved for this action.
 We define the twist of a right $A$-module or an $A$-bimodule similarly.
 For an $A$-bimodule $M$, there are two different right $A$-module
 structures on ${^\gamma}\!M$: the usual one from the right $A$-action
 on $M$ and the twisted one from the left $A$-action on $M$.
 To emphasize the difference, we
 write $(^\gamma\!M)_A$ for the former and ${^\gamma}\! ({_AM})$ for
 the latter.
 If for every simple $A$-module $S$, the $k$-dual of $^{\gamma\!}S$ is
 isomorphic to $S$ as $A$-modules, then we say that \emph{$\gamma$
fixes all simple modules}.

A $k$-algebra will be called $k$-split, if the endomorphism rings of simple
$A$-modules are just $k$. Then simple $A$-bimodules are tensor products of
simple left $A$-modules with simple right $A$-modules.


 \begin{Lem}\label{lem:twist bimodule}
 Let $A$ be a split $k$-algebra with an anti-automorphism $\omega$ fixing
 all simple $A$-modules.  Then:

   $(1)$ For each primitive idempotent $e$ in $A$, the idempotent $\omega(e)$ is conjugate to $e$ in $A$.

   $(2)$ For each projective-injective $A$-module $P$, the $A$-module $\nu_{_A}(P)$ is projective-injective.

   $(3)$ $\domdim(A) = \Ndomdim(A)$. In particular, $A$ is a Morita algebra if and only if $\domdim(A)\geq 2$.

   $(4)$ ${^\omega\!}\!A^{\omega}\cong A$, ${^\omega\!}\DD(A)^\omega \cong \DD(A)$, and
   $\DD({}^{\omega}\! L^{\omega})\cong L$ in $\modcat{A^{\mathrm{e}}}$, for every simple $A$-bimodule $L$.

   $(5)$ Let $e$ be a basic idempotent such that $\add(Ae)=\stp{A}$.    Then in $\Modcat{A^{\mathrm e}}$, $$\Ext^i_{eAe}(eA,eA)\cong \Ext^i_{eAe}(Ae,Ae), \qquad \text{for all $i\geq 0$}.$$

   $(6)$ If $M$ is a non-zero $A$-bimodule with ${}^{\omega}\!M^{\omega}\cong M$
   in $\modcat{A^{\mathrm{e}}}$, then $\Hom_{A^{\mathrm{e}}}(M, \DD(M))\neq 0$.

   $(7)$ If $M$ and $N$ are $A$-bimodules with ${}^{\omega}\!M^{\omega}\cong M$ and
   ${}^{\omega}\! N^{\omega}\cong N$ in $\modcat{A^e}$, then in $\modcat{A^{\mathrm{e}}}$
   $$\Ext^i_A({_A}M,{_A}N)\cong {}^\omega \Ext^i_A(M_A,N_A)^\omega, \qquad \text{for all $i\geq 0$}.$$
   
   $(8)$ If $\cpx{M}$ is a complex in $\D{A^{\rm e}}$ and $m$ is an integer such that:  
   
   \quad\quad {\rm (a)} ${}^{\omega}\!H^m(\cpx{M})^{\omega}\cong H^m(\cpx{M})\neq 0$, and
   
   \quad\quad {\rm (b)}  $H^i(\cpx{M})=0$ for all $i>m$, 
  
 then $m=\max\{d|\Hom_{\D{A^{\rm e}}}(\cpx{M}[d], (\DD\cpx{M})[-d])\neq 0\}$.
  \end{Lem}


 \begin{proof}
  (1) Let $\iota: {}^{\omega^{-1}}\!\!(Ae)\ra \omega(e)A$ be the $k$-linear
  map defined by $\iota(ae)=\omega(e)\omega(a)$, for any $ae\in Ae$.
  Then $\iota(ae\cdot x) = \iota(\omega^{-1}(x)ae)=\omega(e)\omega(a)x$ for
  any $x\in A$. (Note that here $Ae$ gets twisted by $\gamma = \omega^{-1}$.)
  So $\iota$ is a right $A$-module isomorphism.
  Now for each simple left $A$-module $S$,
  \begin{align*}
    \DD(\omega(e)S) \cong \Hom_A(\omega(e)A, \DD(S))
    & \cong \Hom_A({}^{\omega^{-1}}\!\!(Ae), \DD(S)) \\
   &  \cong \Hom_A(Ae, \DD(S)^\omega)
    \cong \Hom_A(Ae, S)\cong eS
  \end{align*}
  by the assumption on $\omega$.
  Therefore, $e$ and $\omega(e)$ are conjugate idempotents in $A$.

  (2) We may assume that $P$ is an indecomposable projective-injective
$A$-module,
  thus of the form $Ae$ for some idempotent $e$. Then $\nu_{_A}(P) = \DD(eA)$
  and $eA\cong \omega(e)A\cong {}^{\omega^{-1}}(Ae)$ by (1). Therefore,
  $eA$ is projective-injective and  $\nu_{_A}(P) \cong \DD(eA)$ is so, too.


 (3) Let $Ae$ be a basic projective-injective $A$-module such that every
 projective-injective $A$-module $P$ belongs to $\add(Ae)$. Then $\nu_{_A}(Ae)\cong Ae$
 as $A$-modules by (2). In other words, every projective-injective $A$-module
 is a strongly projective-injective $A$-module. So by definition,
 $\domdim (A) = \Ndomdim (A)$. The claim about Morita algebras
 then follows from Proposition \ref{prop:characterise morita algebras}.

 (4)
  Let $\theta: A\to {^\omega\!\!}A^\omega$ be the $k$-linear morphism defined
  by: $\theta(a)=\omega(a)$ for $a\in A$. Then for any $x,y\in A$,
  there are equalities
  $\theta(xay)=\omega(y)\omega(a)\omega(x) = x\cdot \theta(a)\cdot y$ in ${^\omega\!}A^\omega$.
  Thus $\theta$ is an $A$-bimodule isomorphism and so is the $k$-dual of $\theta$, that is,
  ${^\omega\!}\DD(A)^\omega\cong  \DD({^\omega\!}A^\omega)\cong \DD(A)$ as $A$-bimodules.
  Since $A$ is a split $k$-algebra, it follows that every simple $A$-bimodule $L$ is
  of the form ${}_AS\otimes \DD(S')_A$, where $S$ and $S'$ are simple left $A$-modules.
  Thus, there is a series of isomorphisms of $A$-bimodules
  \begin{align*}
   \DD({}^{\omega}L^{\omega}) & \cong \DD({}^{\omega}S\otimes \DD(S')^{\omega})
                             \cong \DD({}^{\omega}S\otimes \DD({}^{\omega}S'))
                             \cong \DD({}^{\omega}S\otimes S') \\ &\cong \Hom_k(S', \DD({}^{\omega}S))
                             \cong \Hom_k(S', S)\cong S\otimes \DD(S')\cong L &
  \end{align*}
  as the anti-automorphism $\omega$ fixes simple $A$-modules by assumption.

 (5) Since $\domdim (A) \geq 2$, the minimal faithful $A$-module $Ae$ is basic and
 projective-injective. By (2), $\nu_{_A}(Ae)\cong \DD(eA)$ is also
 basic and projective-injective. As a result, $\DD(eA)\cong Ae$ as $A$-modules, and
 therefore $\DD(eA)_\tau\cong Ae$ as $(A,eAe)$-bimodules for some automorphism
 $\tau$ of $eAe$. So for all $i \geq 0$
 \[
    \Ext^i_{eAe}(eA,eA)\cong \Ext^i_{eAe}(\DD(eA), \DD(eA))\cong
    \Ext^i_{eAe}(\DD(eA)_\tau, \DD(eA)_\tau)\cong \Ext^i_{eAe}(Ae,Ae).
 \]

  (6) Since $M\neq 0$, there exists a simple $A$-bimodule $L$ such that
  $\Hom_{A^{\mathrm{e}}}(M, L)\neq 0$. Therefore
  \begin{align*}
    \Hom_{A^{\mathrm{e}}}(L, \DD(M)) \cong  \Hom_{A^{\mathrm{e}}}(\DD({}^{\omega}\!L^{\omega}), \DD({}^{\omega}\!M^{\omega}))
    \cong  \Hom_{A^{\mathrm{e}}}({}^{\omega}\!M^{\omega}, {}^{\omega}\!L^{\omega})
     \cong \Hom_{A^{\mathrm{e}}}(M, L)\neq 0
  \end{align*}
  Thus the composition of an epimorphism from $M$ to $L$ followed by a monomorphism
  from $L$ to $\DD(M)$ in $\modcat{A^{\mathrm{e}}}$ defines a nonzero morphism
  from $M$ to $\DD(M)$. In particular, $\Hom_{A^{\mathrm{e}}}(M,\DD(M))\neq 0$.

  (7) Since ${}^{\omega}\!M^{\omega}\cong M$ and
   ${}^{\omega}\!N^{\omega}\cong N$ in $\modcat{A^e}$ as $A$-bimodules, it follows
   that
   \[
   \Ext^i_A({}_AM,{}_AN)\cong
   \Ext^i_A({}_A^{\omega}\!M^{\omega},{}_A^{\omega}\!N^{\omega})
   \cong \Ext^i_A({}^{\omega}\!(M_A), {}^{\omega}\!(M_A))
   \cong {}^\omega\Ext^i_A(M_A, N_A)^\omega
   \]
   as $A$-modules for all $i\geq 0$.
   
 $(8)$ By definition, the complex $\cpx{M}[m]$   has vanishing cohomology in degrees
 larger than $0$,  and $(\DD\cpx{M})[-m]$ has vanishing cohomology in degrees smaller than $0$. This implies that $\Hom_{\D{A^{\rm e}}}(\cpx{M}[d], (\DD\cpx{M})[-d])=0$ for all $d>m$. Moreover, by Lemma \ref{lem:morphisms in derived categories}(3),   $\Hom_{\D{A^{\rm e}}}(\cpx{M}[m], (\DD\cpx{M})[-m])$ is isomorphic to $\Hom_{A^{\rm e}}(H^m(\cpx{M}), \DD H^m(\cpx{M}))$, which is nonzero by (6). 
 \end{proof}

 \begin{Theo}[(Second Invariance Theorem)]
   \label{thm:class with anti-automorphisms}
  Let $A$ and $B$ be derived equivalent split $k$-algebras.

  {\rm (a)} If both $A$ and $B$ have anti-automorphisms fixing all simple 
modules, then \mbox{$\gldim(A)=\gldim(B)$}.

  {\rm (b)} If furthermore both $A$ and $B$ have dominant dimension at least 
$1$, then \mbox{$\domdim(A)=\domdim(B)$}.
 \end{Theo}

For the proof of the theorem, we need the following lemma.

\begin{Lem}
Let $A$ be a $k$-split algebra with an anti-automorphism $\omega$ fixing all simple $A$-modules, and let $e$ be an idempotent such that $Ae$ is  basic and $\add(Ae)=\stp{A}$.  Then,  in the canonical triangle
$$\mathbf{(CT2)} \, \, \,  \eta_e^A :\quad  A \lraf{\rho_e}
   \RHom_{eAe}(eA,eA) \lra V_A(e) \lra A[1], $$
for each integer $i$,  we have ${}^{\omega}H^i(V_A(e))^{\omega}\cong H^i(V_A(e))$ as $A$-bimodules.
\label{lem:canonical-complex-VA}
\end{Lem}
\begin{proof}
First of all, it follows from Lemma \ref{lem:twist bimodule}(1) that $A\omega(e)$ and $Ae$ are isomorphic left $A$-modules, and consequently $\omega(e)$ and $e$ are conjugate in $A$, namely, there is an invertible element $u\in A$ such that $\omega(e)=ueu^{-1}$. We fix this notation throughout the proof.

Let $\eta: A\lra \Hom_{eAe}(eA, eA)$ be the canonical map sending each element  $a\in A$ to the map $r_a: ex\mapsto exa$. From the long exact sequence of cohomologies induced by the canonical triangle {\bf CT2}, it follows that 
$$H^i(V_A(e))\cong\begin{cases}
0 & {\mbox{when} \,}  i<-1; \\
 \Ker \eta& {\mbox{when} \,}  i=-1 ; \\
 \Coker\eta & {\mbox{when} \,} i=0; \\
 \Ext_{eAe}^i(eA, eA) & {\mbox{when} \,} i>0.
 \end{cases}$$
For each $i\geq 0$, we have the following isomorphisms of $A$-bimodules
 \begin{align*}
    \Ext^i_{eAe}(eA,eA)&\cong \Ext^i_{eAe}(e{}^\omega\!\!A^\omega, e{}^\omega\!\!A^\omega)
    \cong \Ext^i_{\omega(e)A\omega(e)} ({}^\omega\!\!A\omega(e), {}^\omega\!\!A\omega(e))\\
    &\cong {}^\omega\!\Ext^i_{ueAeu^{-1}}(Aeu^{-1},Aeu^{-1})^\omega
    \cong {}^\omega\! \Ext^i_{eAe}(Ae,Ae)^\omega ,\\
    &\cong {}^\omega\! \Ext^i_{eAe}(eA,eA)^\omega ,\\
 \end{align*} 

\vspace*{-0.5cm}
 
 \noindent 
 where the last isomorphism follows from Lemma \ref{lem:twist bimodule} (5).  Thus, for each integer $i>0$, there is an isomorphism ${}^{\omega}H^i(V_A(e))^{\omega}\cong H^i(V_A(e))$ of $A$-bimodules.
 
It remains to consider the cohomologies of $V_A(e)$ in degrees $0$ and $-1$. For this purpose, we need four $A$-bimodules isomorphisms.  We write down these isomorphisms explicitly, and it is straightforward to check that they are really $A$-bimodule homomorphisms. The first is
$$c_{\omega}: \Hom_{eAe}(eA, eA)\lra {}^{\omega}\Hom_{\omega(e)A\omega(e)}(A\omega(e), A\omega(e))^{\omega}, \, f\mapsto \omega^{-1}f\omega.$$
Note that $\omega(e)=ueu^{-1}$ and $A\omega(e)=Aeu^{-1}$. Let $r_u: Aeu^{-1}\lra Ae$ be the map sending $xeu^{-1}$ to $xe$.  The second isomorphism reads as follows. 
$$c_u: \Hom_{\omega(e)A\omega(e)}(A\omega(e), A\omega(e))\lra \Hom_{eAe}(Ae, Ae), \, g\mapsto r_u^{-1}gr_u.$$
By assumption $Ae$ is basic and $\add(Ae)=\stp{A}$. Then $\nu_AAe$  is again in $\add(Ae)$, basic and has the same number of indecomposable direct summands as $Ae$. Hence $\nu_AAe\cong Ae$ as left $A$-modules.  It follows that $\DD(eA)=\nu_AAe\cong Ae$ as left $A$-modules. Equivalently, $\DD(Ae)\cong eA$ as right $A$-modules, and there is some automorphism $\sigma$ of $eAe$ such that $\DD(Ae)\cong {}_{\sigma}eA$ as $eAe$-$A$-bimodules.  Let $\tau: \DD(Ae)\ra eA$ be such an isomorphism. The third  isomorphism is the canonical map
$$\DD: \Hom_{eAe}(Ae, Ae)\lra\Hom_{eAe}(\DD(Ae), \DD(Ae)),\, h\mapsto \DD(h),$$
and the fourth isomorphism is a composition
$$c_{\tau}: \Hom_{eAe}(\DD(Ae), \DD(Ae))\lra\Hom_{eAe}({}_{\sigma}eA, {}_{\sigma}eA)\lra\Hom_{eAe}(eA, eA),\, h\mapsto  \tau^{-1}h\tau.$$
From the proof of Lemma \ref{lem:twist bimodule} (4), the map $\omega: A\ra {}^{\omega}\!\!A^{\omega}$ is an $A$-bimodule isomorphism. Then it is straightforward to check that the diagram
$$\xymatrix@C=20mm@R=5mm{
A\ar[r]^(.4){\eta}\ar[dddd]^{\omega} & \Hom_{eAe}(eA, eA)\ar[d]^{c_{\omega}}\\
& {}^{\omega}\Hom_{\omega(e)A\omega(e)}(A\omega(e), A\omega(e))^{\omega}\ar[d]^{c_u}\\
& {}^{\omega}\Hom_{eAe}(Ae, Ae)^{\omega}\ar[d]^{\DD}\\
& {}^{\omega}\Hom_{eAe}(\DD(Ae), \DD(Ae))^{\omega}\ar[d]^{c_{\tau}}\\
{}^{\omega}\!\!A^{\omega}\ar[r]^(.4){\eta} & {}^{\omega}\Hom_{eAe}(eA, eA)^{\omega}
} \quad\quad (\star)$$
is commutative. Actually, for each $a\in A$, the image of $a$ under $\eta c_{\omega} c_u \DD$ is $\DD(r_{u}^{-1}\omega^{-1}r_a\omega r_u)$, which sends each $\alpha\in \DD(Ae)$ to $r_{u}^{-1}\omega^{-1}r_a\omega r_u\alpha$. The image of $a$ under $\omega\eta$ is $r_{\omega(a)}$. Then one can check that there is a commutative diagram
$$\xymatrix@C=25mm{
eA\ar[d]^{\tau^{-1}}\ar[r]^{r_{\omega(a)}} & eA\ar[d]^{\tau^{-1}}\\
\DD(Ae)\ar[r]^{\DD(r_{u}^{-1}\omega^{-1}r_a\omega r_u)} & \DD(Ae)
}$$
of right $A$-modules. This means precisely that the image of $a$ under $\eta c_{\omega} c_u \DD c_{\tau}$ is the same as its image under $\omega\eta$. Thus, the diagram $(\star)$ is commutative, and induces another commutative diagram
$$\xymatrix{
0\ar[r] &\Ker\eta\ar[r]\ar@{-->}[d]^{\cong} & A\ar[d]^{\omega}\ar[r]^(.3){\eta} &\Hom_{eAe}(eA, eA)\ar[d]^{c_{\omega}c_u\DD c_{\tau}}\ar[r] & \Coker\eta\ar[r]\ar@{-->}[d]^{\cong} & 0\\
0\ar[r] &{}^{\omega}(\Ker\eta)^{\omega}\ar[r] & {}^{\omega}\!\!A^{\omega}\ar[r]^(.3){\eta} &{}^{\omega}\Hom_{eAe}(eA, eA)^{\omega}\ar[r] & {}^{\omega}(\Coker\eta)^{\omega}\ar[r] & 0
}$$
where the rows are exact and vertical maps are $A$-bimodule isomorphisms. This finishes the proof.
\end{proof}

 \begin{proof}[of Theorem \ref{thm:class with anti-automorphisms}] (a)
 Proposition \ref{prop:bound global dimension} implies that if $A$ or $B$ has
infinite global dimension, then so does the other and we are done.
 Now assume $A$ and $B$ to have finite global dimension
 and $\Db{A}\cong \Db{B}$ to be a standard derived equivalence.
 Let $g$ and $\bar{g}$ be the global dimensions of $A$ and $B$ respectively,
 and let $\cpx{R_{A}}$ denote the bounded complex $\RHom_A({}_A\DD(A), {}_AA)$ in $\Db{A^{\mathrm e}}$,
 and $\cpx{R_B}$ the bounded complex $\RHom_B({}_B\DD(B), {}_BB)$ in
 $\Db{B^{\mathrm{e}}}$. Then Lemma \ref{lem:characterise global dimension}
 can be reformulated as
 \[
    g = \max\{i\mid \mathrm{H}^i(\cpx{R_A})\neq 0\},\qquad \bar{g} = \max\{i\mid \mathrm{H}^i(\cpx{R_B})\neq 0\}.
 \]
 By Lemma \ref{lem:twist bimodule} (4) and (7) and by $k$-duality,
 \[
    \Ext^g_A({_A}\DD(A),{_A}A)\cong {}^\omega\!\Ext^g_A(\DD(A)_A, A_A)^\omega
    \cong {}^\omega\!\Ext^g_A({_A}\DD(A), {_A}A)^\omega
 \]
 as $A$-bimodules. Note that $H^m(\cpx{R_A})\cong\Ext_A^g({}_A\DD(A), {}_AA)$.  Hence, by Lemma \ref{lem:twist bimodule}(8), we get  
 \[
   g = \max\{i \mid \Hom_{\Db{A^{\mathrm{e}}}}(\cpx{R_A}[i], (\DD\cpx{R_A})[-i])\neq 0\}.
 \]
 Similarly $\bar{g} = \max\{i\mid \Hom_{\Db{B^{\mathrm{e}}}}(\cpx{R_B}[i], (\DD\cpx{R_B})[-i])\neq 0\}$.
 Let $\mathcal{F}: \Db{A^{\mathrm{e}}}\cong \Db{B^{\mathrm{e}}}$ be the derived equivalence
 induced from the given standard derived equivalence between $A$ and $B$.
 Then $\mathcal{F}(\cpx{R_A})\cong \cpx{R_B}$ and
 $\mathcal{F}(\DD\cpx{R_A})\cong \DD\cpx{R_B}$ in $\Db{B^{\mathrm e}}$
 by Lemma \ref{lem:standard derived equivalence} (1), (3) and (4).
 Hence, $g=\bar{g}$, that is, $A$ and $B$ have the same global dimension.

 (b) Let $d$ and $\bar{d}$ be
 the $\nu$-dominant dimensions of $A$ and $B$ respectively. Since both $d$ and
 $\bar{d}$ are at least one by assumption, Lemma \ref{lem:twist bimodule} (3)
 (or Lemma \ref{lem:two dominant dimensions coincide}) implies that
 there is no need to distinguish between dominant dimension
 and $\nu$-dominant dimension. Let $e$ and $f$ be basic idempotents in $A$ and $B$, respectively, such that $\add(Ae)=\stp{A}$ and $\add(Bf)=\stp{B}$.  
 Let $\eta_e^A: A\ra \RHom_{eAe}(eA,eA)\ra V_A(e) \ra A[1]$ and 
\mbox{$\eta_f^B: B\ra \RHom_{fBf}(fB,fB)\ra V_B(f) \ra B[1]$} be the canonical
 triangles (denoted by (CT2) in \ref{subsectionstandarddereq})
 associated to the idempotents $e$ and $f$ respectively.

Let $m:=\min\{i|H^i(V_A(e))\neq 0\}$. We claim that $d=m+1$. 
As $d\geq 1$, the module $eA$ is a faithful right $A$-module and $\eta$ is injective. In this case $H^{-1}(V_A(e))=0$, and  $H^0(V_A(e))\cong \Coker\eta$ vanishes if and only if $\eta$ is also surjective, if and only if $d\geq 2$ (see \cite{Mu68}). This implies that $m=0$ when $d=1$. If $d\geq 2$, then it follows from M\"{u}ller's characterisation (Proposition \ref{prop:Muller characterisation}) that $d=m+1$.  

 Without loss of generality, the derived equivalence
 $\D{A}\cong \D{B}$ can be assumed to be standard and to be given
 by a two-sided tilting complex $\cpx{\Delta}$
 in $\D{B\otimes A\opp}$. Let $\mathcal{F}: \D{A^{\mathrm{e}}}\to \D{B^{\mathrm{e}}}$ be the induced standard
 derived equivalence and $\cpx{\Theta} = \RHom_{\D{B}}(\cpx{\Delta}, B)$.
 Then by Theorem \ref{thm:restriction theorem} and Lemma \ref{lem:derived-equivalence-preserves-canonical-triangle},
 \[
    \cpx{\Delta}\otimesL_A \eta_e^A\otimesL_A \cpx{\Theta}\cong \eta_f^B
 \]
 as triangles in $\D{B^{\mathrm{e}}}$, and in particular, $\mathcal{F}(V_A(e))\cong V_B(f)$
 in $\D{B^{\mathrm{e}}}$.
 
 Thus, it follows from Lemma \ref{lem:canonical-complex-VA} and Lemma \ref{lem:twist bimodule}(8) that 
 $$1-d=\max\{i|\Hom_{\D{A^{\rm e}}}((\DD V_A(e))[i], V_A(e)[-i])\neq 0\}, \mbox{  and }$$
 $$1-\bar{d}=\max\{i|\Hom_{\D{B^{\rm e}}}((\DD V_B(f))[i], V_B(f)[-i])\neq 0\}.$$
Since ${\cal F}(V_A(e))\cong V_B(f)$ and ${\cal F}(\DD V_A(e))\cong\DD V_B(f)$, one has $1-d=1-\bar{d}$, and $d=\bar{d}$. 
 \end{proof}



 The characterisations of global and dominant dimension, respectively,
inside the derived category of the enveloping algebras,
 provided by the proof may be of independent interest.

 \begin{Koro}
   Let $A$ be an algebra with an anti-automorphism preserving simples. 
   
   $\mathrm{(a)}$ Let $\cpx{R_{A}}$ denote the bounded
   complex $\RHom_A({}_A\DD(A), {}_AA)$ in $\Db{A^{\mathrm e}}$. Then
   the global dimension of $A$, if finite, is determined in the category
   $\Db{A^{\mathrm e}}$ as $$\gldim(A) = \max\{i \mid
   \Hom_{\Db{A^{\mathrm{e}}}}(\cpx{R_A}[i], (\DD\cpx{R_A})[-i])\neq 0\}. $$
   
 $\mathrm{(b)}$ Suppose $e$ is an idempotent in $A$ such that $Ae$ is basic and $\add(Ae)=\stp{A}$. Then
   the dominant dimension of $A$ is determined
in the category $\Db{A^{\mathrm e}}$ by the canonical triangle
   $$\mathbf{(CT2)} \, \, \,  \eta_e^A :\quad  A \lraf{\rho_e}
   \RHom_{eAe}(eA,eA) \lra V_A(e) \lra A[1], \mathrm{ as}$$
   $$\domdim(A) = 1-\max\{i \mid
   \Hom_{\D{A^{\rm e}}}(\DD (V_A(e))[i], V_A(e)[-i])
   \neq 0\}.$$
 \end{Koro}

\begin{proof}
The proof of Theorem \ref{thm:class with anti-automorphisms} implies that (a) holds, and that (b) holds in case that $\domdim(A)\geq 1$.  Suppose that $\domdim(A)=0$. Then $eA$ cannot be a faithful $A$-module, and thus the canonical map $\eta: A\ra \Hom_{eAe}(eA, eA)$ is not injective. It follows that $H^{-1}(V_A(e))$, which is $\Ker\eta$, is nonzero.  Then, in this case,  (b) follows from Lemma \ref{lem:twist bimodule}(8).
\end{proof}

 
 In the current context, the
 complexes $\cpx{R_{A}}$
used to identify global dimension are preserved at least by
 standard derived equivalences. 

The derived restriction theorem (Corollary 
\ref{cor:restriction theorem for Morita algebras})
and Theorem \ref{thm:restriction theorem} upon which it is based, are
a crucial ingredient of the above proof of Theorem 
\ref{thm:class with anti-automorphisms}. Together with Lemma \ref{lem:derived-equivalence-preserves-canonical-triangle}, they guarantee that the canonical triangle {\bf (CT2)} is preserved under standard derived equivalences between algebras with positive $\nu$-dominant dimensions. 



\begin{Koro}
  Let $A$ and $B$ be algebras with anti-automorphisms preserving simples.
  Then a standard derived equivalence from $A$ to $B$ induces a standard
equivalence from  $\Db{A^{\mathrm e}}$ to  $\Db{B^{\mathrm e}}$ that
sends $\cpx{R_{A}}$ to $\cpx{R_{B}}$.
 \end{Koro}

When an algebra $A$ with an anti-automorphism preserving simples is derived
equivalent to an algebra $B$, possibly without such an anti-automorphism,
the proof of Theorem \ref{thm:class with anti-automorphisms} provides an
inequality:

\begin{Koro} \label{cor:inequality}
 Let $A$ be an algebra with an anti-automorphism preserving simples and
suppose $A$ is derived equivalent to an algebra $B$. Then there is an
inequality \mbox{$\gldim(B) \geq \gldim(A)$}.
\end{Koro}

\begin{proof}
The characterisation of $\gldim(A) = g$ given in the proof of Theorem
\ref{thm:class with anti-automorphisms} remains valid. Then
the non-vanishing over $A$ of
\mbox{$0 \neq \Hom_{\Db{A^{\mathrm{e}}}}(\cpx{R_A}[g], (\DD\cpx{R_A})[-g]) \cong$}
$\Hom_{\Db{B^{\mathrm{e}}}}(\cpx{R_B}[g], (\DD\cpx{R_B})[-g])\neq 0\}$,
implies a contradiction to $\gldim(B) < g$.
\end{proof}


The following example shows that the assumption in the Second Invariance
Theorem \ref{thm:class with anti-automorphisms}, requiring both
algebras admitting an anti-automorphism preserving simples, cannot be relaxed;
the inequality in Corollary \ref{cor:inequality} is in general not an equality.
In this example, the algebra $A$ has an anti-automorphism, while $B$ does not.
The algebras $A$ and $B$ are derived equivalent. The global dimension of $B$
is strictly bigger than that of $A$.

Let the algebra $A$ be given by the quiver
$$
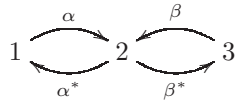
\xy<40pt, 0pt>:
(0, 0)*+{1}="1",
(1,0)*+{2}="2",
(2,0)*+{3}="3",
{\ar@/^.7pc/^{\alpha} "1";"2"},
{\ar@/^.7pc/^{\alpha^*} "2";"1"},
{\ar@/_.7pc/_{\beta} "3";"2"},
{\ar@/_.7pc/_{\beta^*} "2";"3"},
\endxy$$
with relations $\{\alpha^*\alpha, \beta^*\beta, \beta^* \alpha, \alpha^* \beta
\}$. This is a {\em dual extension} (defined by C.C.Xi \cite{Xi})
of the path algebra of $1\rightarrow 2\leftarrow 3$. It is a
quasi-hereditary cellular algebra, and has global dimension $2$. Let $P_i$ be
the indecomposable projective $A$-modules corresponding to the vertex $i$, and
let $T^{\bullet}$ be the direct sum of three indecomposable complexes
${P_1^{\bullet}}$, $P_2^{\bullet}$ and $P_3^{\bullet}$, where
$P_2^{\bullet}:=P_2[1]$ and, for each $i\in\{1, 3\}$, the complex
$P_i^{\bullet}$ is the complex $$0\rightarrow P_2\rightarrow P_i\rightarrow 0$$
with $P_2$ in degree $-1$. Then $T^{\bullet}$ is a tilting complex, inducing
a derived equivalence between $A$ and the
endomorphism algebra $B$ of $T^{\bullet}$. The algebra $B$ has the following
quiver
$$
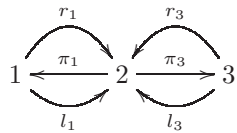
\xy<40pt, 0pt>:
(0, 0)*+{1}="1",
(1,0)*+{2}="2",
(2,0)*+{3}="3",
{\ar_{\pi_1} "2";"1"},
{\ar^{\pi_3} "2";"3"},
{\ar@/^1pc/^{l_3} "3";"2"},
{\ar@/_1pc/_{l_1} "1";"2"},
{\ar@/_1.5pc/_{r_3} "3";"2"},
{\ar@/^1.5pc/^{r_1} "1";"2"},
\endxy$$
with relations $\{l_1\pi_1-l_3\pi_3, \,  r_1\pi_1-r_3\pi_3, \, \pi_1l_i,
\pi_3r_i, i=1, 3\}$. The algebra $A$ has an anti-automorphism fixing simple
modules, while $B$ cannot have such an anti-automorphism. (If an algebra
$\Lambda$ has such an anti-automorphism, then
$\Ext_{\Lambda}^1(U, V)\cong\Ext_{\Lambda}^1(V, U)$ for all simple
$\Lambda$-modules $U, V$. This implies that, in the quiver of the algebra
$\Lambda$, for any two vertices $a$ and $b$, the number of arrows
from $a$ to $b$ coincides with the number of arrows in the opposite direction,
from $b$ to $a$). A
direct calculation shows that the algebra $B$ has global dimension $3$, which
is strictly larger than $\gldim(A)$.
\bigskip

Cellular algebras by definition have anti-automorphisms preserving simples;
hence this large and widely studied class of algebras is covered by the 
Second Invariance Theorem \ref{thm:class with anti-automorphisms}.

\begin{Koro} \label{cor:cellular}
Suppose that $A$ and $B$ are derived equivalent cellular algebras
over a field. Then $A$ and $B$ have the same global dimension. If both $A$ and 
$B$ have positive dominant dimension, then their dominant dimensions are equal. 
\end{Koro}

In the next Section, it will be shown how to use this result to compute
homological dimensions of blocks of quantised Schur algebras.

\section{Homological dimensions of $q$-Schur algebras and their blocks}
\label{sec:two dimensions for q-schur algebras}
  Global dimensions of classical and quantised Schur algebras $S(n,r)$ (with $n \geq r$)
  have been determined by Totaro \cite{T97} and Donkin \cite{D98}. Dominant
  dimensions of these algebras have been obtained more recently in
  \cite{FM,FK11a}. The aim of this section is to determine the
  global and dominant dimensions of all blocks of these algebras, that is, of the
  indecomposable algebra direct summands.

  To recall some basics on $q$-Schur algebras and their blocks (see \cite{D98}
  and the references therein),
  let $k$ be a field of characteristic $0$ or $p$, and $q$
  a non-zero element in $k$. Let $\ell$ be the smallest integer
  such that $1+q+\cdots+q^{\ell-1}=0$, and set $\ell=0$ if no such
  integer exists.
  For a natural number $r$, let $\Sigma_r$ be the symmetric group
  on $r$ letters and let $\H_q(r)$ be the associated Hecke algebra that
  is defined by generators $\{T_1,\ldots, T_{r-1}\}$ and relations
  \begin{align*}
    (T_i+1)(T_i-q)& =0, \qquad &&(1\leq i\leq r-1);\\
    T_iT_j & =T_jT_i, \qquad &&(|i-j|>1);\\
    T_iT_{i+1}T_i& = T_{i+1}T_iT_{i+1},\qquad &&(1\leq i\leq r-2).
  \end{align*}
  For each composition $\lambda=(\lambda_1,\ldots, \lambda_n)$ of $r$, that
  is, a sequence of $n$ non-negative integers summing up to $r$,
  let $\H_q(\Sigma_\lambda)$ be the associated parabolic Hecke algebra
  that is isomorphic to $\H_q(\lambda_1)\otimes_k\cdots \otimes_k \H_q(\lambda_n)$
  as $k$-algebras. Then $\H_q(\Sigma_\lambda)$ naturally can be seen
  as a $k$-subalgebra of $\H_q(r)$. The algebra
  $\H_q(\Sigma_\lambda)$ has a trivial module $k$ with all $T_i$'s acting
  as scalar $q$. Let $M^\lambda$ be the permutation module over $\H_q(r)$
  that is defined as the induced module $\H_q(r)\otimes_{\H_q(\Sigma_\lambda)} k$.
  The $q$-Schur algebra $S_q(n,r)$ is then
  defined to be the endomorphism ring $\End_{\H_q(r)}(\bigoplus_{\lambda}M^\lambda)$
  where $\lambda$ ranges over all compositions $(\lambda_1,\ldots, \lambda_n)$ of $r$ into $n$ parts, where $n$ is any natural number.
  For each composition $\lambda$ of $r$, there is a unique associated
  composition $\overline{\lambda}$ (called \emph{partition}) obtained by rearranging the entries of $\lambda$
  weakly decreasing; the permutation module $M^{\overline{\lambda}}$ is
  isomorphic to $M^\lambda$ and is a direct sum of the (indecomposable) Young modules
  $Y^\mu$ with multiplicities
  $K_{\overline{\lambda},\mu}$ ($p$-Kostka number) where $\mu$ ranges over all partitions of $r$. These multiplicities are known to satisfy
  $K_{\overline{\lambda},\overline{\lambda}}=1$ and $K_{\overline{\lambda},\mu}=0$ unless $\mu\trianglerighteq\overline{\lambda}$
  in the dominance ordering on partitions.

  When $q$ is not a root of unity, then the $q$-Schur algebra $S_q(n,r)$ is semisimple,
  and thus the global dimensions of $S_q(n,r)$ and its blocks are $0$, and the dominant
  dimensions of $S_q(n,r)$ and its blocks are $\infty$. In the following,
  we always assume that $q$ is a root a unity, and therefore $\ell>0$.
  If $p=0$ and $r = r_{-1}+\ell r_0$ is the $\ell$-adic expansion of $r$, then we set
  $d_{\ell,p}(r) = r_{-1}+r_0$; if $p>0$, $r=r_{-1}+\ell r'$ and $r' = r_0+pr_1+p^2r_2+\cdots$
  are the $\ell$-adic expansion of $r$ and the $p$-adic expansion of $r'$ respectively, then
  we set $d_{\ell, p}(r)=r_{-1}+r_0+r_1+r_2+\cdots$.
  The global dimension of $S_q(n,r)$ for $n\geq r$ in this case has been given
  by Totaro \cite{T97} for $q=1$, that is for the classical Schur algebra, and by Donkin \cite{D98} in general.

  \begin{Theo}[(\cite{D98,T97})]\label{thm:global dimension of Schur algebras}
If $q$ is a root of unity and $n\geq r$, then 
\mbox{$\gldim S_q(n,r)=2(r-d_{\ell, p}(r))$}.
  \end{Theo}

  A lower bound for the dominant dimension of $S_q(n,r)$ has been obtained (implicitly)
  by Kleshchev and Nakano \cite{KN01} for $q=1$ and by Donkin
  \cite[Proposition 10.5]{D07} in general. It was shown later in
  \cite{FM,FK11a} that this lower bound is an upper bound, too.

  \begin{Theo}[(\cite{D07,FM,FK11a,KN01})] 
\label{thm:dominant dimension of Schur algebras}
    If $q$ is a root of unity and $n\geq r$, then 
\mbox{$\domdim S_q(n,r)=2(\ell-1)$}.
  \end{Theo}

  To determine the global and dominant dimensions of blocks of the
  $q$-Schur algebras, we will use the setup of algebras with a duality.
  Each $q$-Schur algebra has an anti-automorphism
  that fixes all simple modules, and there is a block decomposition
 that is invariant under the involution
  \[
    S_q(n,r)\cong \bigoplus_{(\tau,w)} \mathbf{B}_{\tau,w}
  \]
  where $0\leq w\leq r$ and $\tau$ ranges over all $\ell$-core partitions of $r-w\ell$.
  Moreover, for $m, n\geq r$, the two blocks $\mathbf{B}_{\tau, w}$ of $S_q(n,r)$ and
  $\mathbf{B}_{\tau',w'}$ of $S_q(m,r)$ are derived equivalent if $w=w'$ \cite{CR08}.

 \begin{Theo} \label{thm:global dimension of blocks of Schur algebras}
  If $q$ is a root of unity and $n\geq r$, then the global dimension of
  the block $\mathbf{B}_{\tau,w}$ is equal to
  $2(\ell w-d_{\ell,p}(\ell w))$.
 \end{Theo}

 \begin{proof}
  Chuang and Rouquier have shown in \cite{CR08}, that $\mathbf{B}_{\tau, w}$ and $\mathbf{B}_{\emptyset, w}$ are derived equivalent. As all block
 algebras of $q$-Schur algebras have anti-automorphisms fixing all simple modules, Theorem \ref{thm:class with anti-automorphisms} can be applied and we get
 $\gldim\mathbf{B}_{\tau, w} = \gldim\mathbf{B}_{\emptyset, w}$.

 For each natural number $s$, set $g(s) = 2(\ell s- d_{\ell,p}(\ell s))$. Then
 $g(s+1) =$ \mbox{$2(\ell s+ \ell - d_{\ell, p}(\ell s+ \ell))$} $=
g(s)+2(\ell+d_{\ell,p}(\ell s)-d_{\ell,p}(\ell s+ \ell))\geq
 g(s)+2(\ell-1)$. In particular, $g(s)>g(s')$ if $s>s'$.

 Now we are going to compute the global dimension of $\mathbf{B}_{\tau,w}$
by induction on $w$; we have to show that it equals $g(w)$.
 If $w=0$, then the block algebra $\mathbf{B}_{\tau,w}$ is semisimple
(\cite{D98}),
 and thus $\gldim \mathbf{B}_{\tau,w}=0=g(w)$. Now we assume that $w>0$.
 Note that the $q$-Schur algebra $S_q(\ell w, \ell w)$ is of global dimension
 $g(w)$ by Theorem \ref{thm:global dimension of Schur algebras}, and has a block
 subalgebra $\mathbf{B}_{\emptyset,w}$. It follows that $\gldim \mathbf{B}_{\emptyset,w}=g(w)$
 since all other block subalgebras of $S_q(\ell w, \ell w)$ are
 $\mathbf{B}_{\tau,w'}$ with $w'<w$, and thus $\gldim \mathbf{B}_{\tau,w'}
= g(w') < g(w)= \gldim \mathbf{B}_{\emptyset, w}$
 by induction.
 \end{proof}

 In terms of the cover theory introduced by Rouquier \cite{Ro08},
 the $q$-Schur algebra $S_q(n,r)$ is a quasi-hereditary cover of the Hecke algebra
 $\H_q(r)$ of covering degree $(\ell-1)$ 
by Theorem \ref{thm:dominant dimension of Schur algebras} and \cite{FK11a},
 that is, $S_q(n,r)$ is a $(\ell-1)$-cover, 
but not an $\ell$-cover of $\H_q(r)$.
 The following result implies a particular property of the cover;
each block of $S_q(n,r)$
 is a quasi-hereditary cover of the corresponding block of $\H_q(r)$, of the
same dominant dimension. This property may be formulated as saying that the
covering is uniform of covering degree $\ell-1$.

 \begin{Theo} \label{thm:dominant dimension of blocks of Schur algebras}
 If $q$ is a root of unity and $n\geq r$, then the dominant dimension of 
$\mathbf{B}_{\tau, w}$ satisfies
 $$
 \domdim \mathbf{B}_{\tau, w} =
 \begin{cases}
 \infty & {\mbox{when} \,}  w=0 ; \\
 2(\ell-1) & {\mbox{when} \,} w \neq 0.
 \end{cases}$$
 \end{Theo}

 \begin{proof}
 If $w=0$, then $\mathbf{B}_{\tau, w}$ is semisimple and has dominant dimension infinity. Now we assume that $w>0$. Since $q$ is a root of unity, it follows that $\ell\geq 2$, and thus by Theorem \ref{thm:dominant dimension of Schur algebras}
 \begin{align} \label{eqn:lower bound for dominant dimension}
 \domdim \mathbf{B}_{\tau, w}\geq \domdim S_q(n, r)= 2(\ell-1)\geq 2.
 \end{align}
Note that  all block algebras have anti-automorphisms fixing all simple modules.  By \cite{CR08}, the block subalgebra $\mathbf{B}_{\tau,w}$ of $S_q(n,r)$ and
 the principal block subalgebra $\mathbf{B}_{\emptyset,w}$ of $S_q(\ell w, \ell w)$ are
 derived equivalent, and thus have the same dominant dimension by Theorem \ref{thm:class with anti-automorphisms}.
 Therefore we only need to show $\domdim \mathbf{B}_{\emptyset,w}=2(\ell-1)$.

Let $e$ be an idempotent in $\mathbf{B}_{\emptyset,w}$ such that $\mathbf{B}_{\emptyset,w}e$
 is a minimal faithful $\mathbf{B}_{\emptyset,w}$-module. Then $\mathbf{b}_{w}=e\mathbf{B}_{\emptyset,w}e$
 is a block subalgebra of $\H_q(\ell w)$ and the $\mathbf{b}_w$-module $e\mathbf{B}_{\emptyset, w}$ is isomorphic to
 a direct sum of those Young modules $Y^\mu$ that belong to the block $\mathbf{b}_w$ (see \cite{D98}).
 By Proposition \ref{prop:Muller characterisation} and the inequality (\ref{eqn:lower bound for dominant dimension}),
 to finish the proof, it suffices to show:

\smallskip
 {\it Claim. There exist Young modules $Y^\lambda$ and $Y^\mu$ of $\H_q(\ell w)$ that belong to $\mathbf{b}_w$
 such that $\Ext^{2(\ell-1)-1}_{\mathbf{b}_w}(Y^\lambda,Y^\mu)\neq 0$. }

 \smallskip
 {\it Proof.} Set $\nu = (\ell, 1,\ldots, 1)$ and $\mu=(\ell w)$, which are two partitions of $\ell w$.
 Then the Young module $Y^\mu$ belongs to the block $\mathbf{b}_w$ and by definition $Y^\mu = M^\mu = k$.
 As a result, by Mackey's decomposition theorem
 \[
    \Ext^{2(\ell-1)-1}_{\H_q(r)}(M^\nu, Y^\mu) \cong \Ext^{2(\ell-1)-1}_{\H_q(r)}(\H_q(r)\otimes_{\H_q(\Sigma_\nu)} k, k)
    \cong
    \Ext^{2(\ell-1)-1}_{\H_q(\Sigma_\nu)}(k,k)\neq 0.
\]
  Here, the first isomorphism uses the definition
of permutation modules and the second one uses adjointness; non-vanishing
of the third extension space follows by identifying $\Sigma_\nu$ with
$\Sigma_{\ell}$ and then using the known cohomology of $\H_q(\Sigma_{\ell})$. So there is a direct summand $Y^\lambda$ of $M^\nu$ such that
 $\Ext^{2(\ell-1)-1}_{\H_q(r)}(Y^\lambda, Y^\mu)\neq 0$. In this case, the Young module
 $Y^\lambda$ must belong to the block $\mathbf{b}_w$, too. This proves the
claim and the Theorem.
 \end{proof}

 \noindent \textit{Remark.}
 (1) When $k$ has characteristic zero or bigger than the weight $w$,
 the dominant dimension of the blocks $\mathbf{B}_{\tau, w}$
 have been determined in \cite{FM} by using Chuang and Tan's complete
 description of the corresponding Rouquier blocks \cite{CT03}.

 (2) All blocks $\mathbf{B}_{\tau,2}$ of the $q$-Schur algebra $S_q(n,r)$
 are also stably equivalent of Morita type, and hence have the same global and
 dominant dimensions, as well as the same representation dimension. Indeed,
 by carefully examining the tilting complexes constructed by
 Chuang and Rouquier \cite{CR08},
 we see that the induced derived equivalences are almost $\nu$-stable.
 However, the derived equivalences between general blocks constructed by
 Chuang and Rouquier \cite{CR08}
 are not almost $\nu$-stable. For instance,  when $p=2$,  the group algebras
 of $S_6$ and of $S_7$, respectively, each have a unique block of $p$-weight
 $3$. These two blocks correspond to each other under the reflection $s_0$ of
 the Weyl group of the affine Kac-Moody algebra $\widehat{sl_2}$. The derived
 equivalences constructed in \cite{CR08}, relating these two blocks, are not
 almost $\nu$-stable.

\affiliationone{
   Ming Fang

   \smallskip
   Institute of Mathematics\\
   Chinese Academy of Sciences\\
   100190 Beijing, P.R.China
   \email{fming@amss.ac.cn}}   

\affiliationone{
   Wei Hu

   \smallskip
   School of Mathematical Sciences\\
   Beijing Normal University\\
   100875 Beijing, P.R.China
   \email{huwei@bnu.edu.cn}}
   
\affiliationone{
   Steffen Koenig

   \smallskip
   Institute of Algebra and Number Theory\\
  University of Stuttgart\\
   Pfaffenwaldring 57\\
   70569 Stuttgart, Germany
   \email{skoenig@mathematik.uni-stuttgart.de}}

\vspace{-.5cm}


\begin{thebibliography}{99}

  \bibitem{AI} T. Aihara and O. Iyama,
  Silting mutation in triangulated categories,
  \textit{J. Lond. Math. Soc. (2)} \textbf{85} (2012), 633-668.

  \bibitem{AR04} S. Al-Nofayee and J. Rickard,
  Rigidity of tilting complexes and derived equivalences for
  self-injective algebras, \textit{preprint}.

  \bibitem{AF} F.W. Anderson and K.R. Fuller,
    \textit{Rings and categories of modules},
    Graduate Texts in Mathematics \textbf{13},
  Springer-Verlag, New York, second edition, 1992.

\bibitem{Asa} H. Asashiba, The derived equivalence classification of
  representation--finite selfinjective algebras.
  \textit{J. Algebra} \textbf{214} (1999), 182--221.

\bibitem{ASS} I. Assem, D. Simson and A. Skowronski,
  \textit{Elements of the representation theory of associative algebras.} 
  Vol. 1. Techniques of representation theory. London Mathematical Society
  Student Texts, \textbf{65}. Cambridge University Press, Cambridge, 2006.
  
  \bibitem{AS} M. Auslander and S.O. Smal{\o},
  Preprojective modules over artin algebras,
  \textit{J. Algebra} \textbf{66} (1980), 61-122.

  \bibitem{ChM}A. Chan and R. Marczinzik, On gendo-Brauer tree algebras. preprint.

  \bibitem{CR08} J. Chuang and R. Rouquier,
  Derived equivalences for symmetric groups and $\mathfrak{sl}_2$-categorification,
  \textit{Ann. Math.} \textbf{167} (2008), 245-298.

  \bibitem{CT03} J. Chuang and K.M. Tan,
  Filtrations in Rouquier blocks of symmetric groups and Schur algebras,
  \textit{Proc. London Math. Soc.} \textbf{86} (2003), 685-706.

  \bibitem{CPS96} E. Cline, B. Parshall and L. Scott,
  Stratifying endomorphism algebras,
  \textit{Mem. Amer. Math. Soc.} \textbf{124} (1996).

  \bibitem{D98} S. Donkin,
  \textit{The $q$-Schur algebra},
  London Mathematical Society Lecture Note Series \textbf{253},
  Cambridge University Press, Cambridge, 1998.

  \bibitem{D07} S. Donkin,
  Tilting modules for algebraic groups and finite dimensional
  algebras, in \textit{Handbook of Tilting Theory},
  London Mathematical Society Lecture Note Series \textbf{332}, Cambridge
  University Press, Cambridge, 2007.

  \bibitem{DugasMV} A. Dugas and R. Martinez-Villa,
  A note on stable equivalences of Morita type. \textit{J. Pure Appl.
    Algebra} \textbf{208} (2007), 421-433.

  \bibitem{FKY} M. Fang, O. Kerner and K. Yamagata,
  Canonical bimodules and dominant dimension, to appear in
 \textit{Trans. Amer. Math. Soc.}

  \bibitem{FK11a} M. Fang and S. Koenig,
  Schur functors and dominant dimension,
  \textit{Trans. Amer. Math. Soc.} \textbf{363} (2011), 1555-1576.
 
  \bibitem{FK11b} M. Fang and S. Koenig,
  Endomorphism algebras of generators over symmetric algebras,
  \textit{J. Algebra} \textbf{332} (2011), 428-433.
  
  \bibitem{FK15} M. Fang and S. Koenig,
  Gendo-symmetric algebras, canonical comultiplication, Hochschild cocomplex
  and dominant dimension,
 \textit{Trans. Amer. Math. Soc.} \textbf{368} (2016) 5037-5055.

\bibitem{FM} M. Fang and H. Miyachi,
  Dominant dimension, Hochschild cohomology and derived equivalence,
  \textit{preprint}.

 \bibitem{GR} P. Gabriel and A.V. Roiter,
  \textit{Representations of finite dimensional algebras,}
  Encyclopedia Math. Sci. Vol \textbf{73}, Springer-Verlag, 1992.
Chapter 12, Derivation and Tilting, by B. Keller.

  \bibitem{Ha88} D. Happel,
  \textit{Triangulated categories in the representation theory of finite dimensional algebras},
  London Math. Soc. Lecture Note Series \textbf{119} (1988), Cambridge University Press.

  \bibitem{Hu12} W. Hu
  On iterated almost $\nu$-stable derived equivalences,
  \textit{Comm. Algebra} \textbf{40}(2012), 3920-3932.

  \bibitem{HuKXi13} W. Hu, S. Koenig, and C.C. Xi,
  Derived equivalences from cohomological approximations and mutations
  of $\Phi$-Yoneda algebras, \textit{Proc. Roy. Soc. Edinburgh Sect. A}
  \textbf{143} (2013), 589-629.

  \bibitem{HuXi10} W. Hu and C.C. Xi,
  Derived equivalences and stable equivalences of Morita type, \textrm{I}.
  \textit{Nagoya Math. J.} \textbf{200} (2010),107-152.

  \bibitem{HuXi11} W. Hu and C.C. Xi,
  $\mathscr{D}$-split sequences and derived equivalences,
  \textit{Adv. Math.} \textbf{227} (2011), 292-318.


  \bibitem{HuXi13} W. Hu and C.C. Xi,
  Derived equivalences for $\Phi$-Auslander-Yoneda algebras,
  \textit{Trans. Amer. Math. Soc.} \textbf{365} (2013), 5681-5711.

  \bibitem{HuXi} W. Hu and C.C. Xi,
  Derived equivalences and stable equivalences of Morita type, \textrm{II},
  arXiv:1412.7301.

\bibitem{Ke94}
B.~Keller, Deriving DG categories,  \emph{Ann. Sci.
  {\'E}cole Norm. Sup.} \textbf{27} (1994), 63--102.

  \bibitem{Ke96} B. Keller,
  {Derived categories and their uses}, in
  {\it Handbook of algebra}, Amsterdam, North Holland, 1996.

  \bibitem{KY13} O. Kerner and K. Yamagata,
  Morita algebras, \textit{J. Algebra} \textbf{382} (2013), 185-202.

  \bibitem{KN01} A.S. Kleshchev and D.K. Nakano,
  On comparing the cohomology of general linear groups and symmetric groups,
  \textit{Pacific J. Math.} \textbf{201} (2001), 339-355.

\bibitem{KSX} S. Koenig, I.H. Slung{\aa}rd and C.C. Xi,
  Double centralizer properties, dominant dimension, and tilting modules,
  \textit{J. Algebra} \textbf{240} (2001), 393--412.

  \bibitem{LXi05} Y.M. Liu and C.C. Xi,
  Constructions of stable equivalences of Morita type for finite dimensional algebras \textrm{II},
  \textit{Math. Z.} \textbf{251} (2005), 21-39.

\bibitem{MV1} R. Martinez-Villa, The stable equivalence for algebras
of finite representation type,
\textit{Commun. Alg.} \textbf{13} (1985), 991-1018.

\bibitem{MV2} R. Martinez.Villa,
Properties that are left invariant under stable equivalences,
\textit{Commun. Alg.} \textbf{18}, 4141-4169.

  \bibitem{MO04} V. Mazorchuk and S. Ovsienko,
  Finitistic dimension of properly stratified algebras,
  \textit{Adv. Math.} \textbf{186} (2004), 251-265.


  \bibitem{Mi03} J. Miyachi,
  Recollement and tilting complexes,
  \textit{J. Pure Appl. Algebra} \textbf{183} (2003), 245-273.


  \bibitem{Mu68} B. M\"{u}ller,
  The classification of algebras by dominant dimension,
  \textit{Canad. J. Math.} \textbf{20} (1968), 398-409.


  \bibitem{Mo58} K. Morita,
  Duality for modules and its applications to the theory of rings with minimum condition,
  \textit{Sci. Rep. Tokyo Kyoiku Daigaku Sec. A} \textbf{6} (1958), 83-142.


  \bibitem{R89} J. Rickard,
  Morita theory for derived categories,
  \textit{J. London Math. Soc.} \textbf{39} (1989), 436-456.

\bibitem{Rickardderivedstable} J. Rickard,
Derived categories and stable equivalence,
\textit{J. Pure Appl.Alg.} \textbf{61} (1989), 303-317.

  \bibitem{R91} J. Rickard,
  Derived equivalences as derived functors,
  \textit{J. London Math. Soc.} \textbf{43} (1991), 37-48.

  \bibitem{Ro08} R. Rouquier,
  $q$-Schur algebras and complex reflection groups,
  \textit{Moscow Math. J.} \textbf{8} (2008), 119-158.


  \bibitem{Sp85} N. Spaltenstein,
  Resolutions of unbounded complexes,
  \textit{Comp. Math.} \textbf{65} (1988), 121-154.

  \bibitem{T97} B. Totaro,
  Projective resolutions of representations of $GL(n)$,
  \textit{J. reine angew. Maht.} \textbf{482} (1997), 1-13.

\bibitem{Xi} C.C. Xi,
  Quasi-hereditary algebras with a duality.
  \textit{J. Reine Angew. Math.} \textbf{449} (1994), 201-215.

  \bibitem{Y96} K. Yamagata,
  Frobenius algebras, in \textit{Handbook of algebras},
  North-Holland, Amsterdam 1996, 841-887.


\end{thebibliography}
\end{document}